\documentclass[11pt,reqno]{amsart}
\usepackage[margin=1.0in]{geometry}
\usepackage[utf8]{inputenc}
\usepackage[T1]{fontenc}
\usepackage{lmodern}
\usepackage[english]{babel}
\usepackage{amsmath,a4wide}
\usepackage{xfrac}
\usepackage{esint}
\usepackage{bbm}
\usepackage{mathrsfs,bm,amsthm,mathtools,yfonts,amssymb,color,braket,booktabs,graphicx,graphics,amsfonts,comment,siunitx}
\usepackage{url}
\usepackage[square,numbers]{natbib}
\usepackage{hyperref}
\usepackage{cleveref}

\newcommand{\R}{\mathbb{R}}

\newcommand{\N}{\mathbb{N}}
\newcommand{\V}{\mathbb{V}}

\newcommand{\lv}{\lvert}
\newcommand{\rv}{\rvert}
\newcommand{\lV}{\lVert}
\newcommand{\rV}{\rVert}

\DeclareMathOperator{\spt}{spt}

\DeclareMathOperator{\tr}{tr}
\DeclareMathOperator{\dist}{dist}
\DeclareMathOperator{\dv}{div}
\DeclareMathOperator{\diam}{diam}

\theoremstyle{plain}
\newtheorem{theorem}{Theorem}[section]
\newtheorem*{theorem*}{Theorem}
\newtheorem{lemma}[theorem]{Lemma}
\newtheorem*{lemma*}{Lemma}
\newtheorem{prop}[theorem]{Proposition}
\newtheorem*{prop*}{Proposition}
\newtheorem{corollary}[theorem]{Corollary}
\newtheorem*{corollary*}{Corollary}
\newtheorem{assumption}[theorem]{Assumption}
\newtheorem*{assumption*}{Assumption}
\theoremstyle{definition}

\newtheorem*{example*}{e.g.}
\newtheorem{remark}[theorem]{Remark}
\newtheorem*{remark*}{Remark}
\numberwithin{equation}{section}
\newtheorem{definition}[theorem]{Definition}
\newtheorem*{definition*}{Definition}

\title{Varifold convergence of free boundary Allen--Cahn equation}



\author{Jingeon An-Lacroix}

\address{Department of Mathematics and Computer Science, University of Basel, Spiegelgasse 1, 4052 Basel, Switzerland}

\email{jingeon.an\_lacroix@icloud.com}

\author{Kiichi Tashiro}

\address{Department of Mathematics, Institute of Science Tokyo, 2-12-1, Ookayama, Meguro-ku, Tokyo, 152-8551, Japan}

\email{tashiro.k.0e2f@m.isct.ac.jp}

\begin{document}
	
\begin{abstract}
The free boundary Allen--Cahn equation
    \[
        \left\{
        \begin{alignedat}{2}
            \Delta u&=0\quad&&\text{in}\quad\{|u|<1\}\\
            |\nabla u|&=1/\varepsilon\quad&&\text{on}\quad\partial\{|u|<1\},
        \end{alignedat}
        \right.
    \]
has recently attracted considerable attention because it retains the essential features of the classical Allen--Cahn equation while being significantly more tractable.

In this work, we establish the free boundary analogue of the seminal Hutchinson--Tonegawa theory, developing the varifold convergence framework for solutions of the free boundary Allen--Cahn equation to minimal surfaces. In addition, we provide the $\Gamma$-convergence of the free boundary Allen--Cahn energy to the area functional, and the conservation of local minimization property. This foundation is expected to be used in further applications of the free boundary Allen--Cahn equation in the study of minimal surfaces, such as providing an alternative proof of celebrated Yau's conjecture, possibly with simpler and more complete arguments.
\end{abstract}
	
	\maketitle
	
\section{Introduction}

In this work, we study the free boundary Allen--Cahn equation
\[
    \left\{
        \begin{alignedat}{2}
            \Delta u&=0\quad&&\text{in}\quad\{|u|<1\}\\
            |\nabla u|&=1/\varepsilon\quad &&\text{on}\quad\partial\{|u|<1\}.
        \end{alignedat}
    \right.
\]
This overdetermined problem is derived from Ginzburg-Landau energy functional
\[
    J_\varepsilon(u):=\int_{\Omega}\left(\varepsilon|\nabla u|^2+\frac{\chi_{(-1,1)}(u)}{\varepsilon}\right)\,dx,
\]
where $\varepsilon>0$ is a small parameter determining the thickness of the interface $\{|u|<1\}$. In its classical form, the Allen--Cahn equation is given by
\begin{equation}\label{eq: intro, AC}
    \Delta u(x)=W'(u)/\varepsilon^2,
\end{equation}
where $W:[-1,1]\rightarrow\mathbb{R}$ is the double--well potential that attains $0$ at $\pm 1$, and strictly positive in $(-1,1)$. Prominent examples of such potentials \( W \) are given by the family of functions $\{W_{\delta}\}_{0\le \delta\le 2}$, 
\[ W_\delta(u) :=
\begin{cases}
(1 - u^2)^\delta &\qquad\text{for}\quad 0 < \delta \le 2,\\
\chi_{(-1,1)}(u)& \qquad \text{for}\quad \delta = 0,
\end{cases}\]
which give rise to the parametrized set of energy functionals
\begin{equation} \label{eq:energy-functional}
    J^\delta_{\varepsilon}(u;\Omega) := \int_\Omega \left( \varepsilon |\nabla u|^2 + \frac{ W_\delta(u) }{\varepsilon}\right) \, dx,\quad\text{for}\quad \delta\in [0, 2].
\end{equation}

The Allen--Cahn equation was first introduced by the physical motivation in \cite{allen1972ground,allen1973correction}, to describe phase transition models such as binary alloy. However, it is also of great interest from a mathematical perspective, due to its property that closely resembles minimal surfaces (see also \cite{caffarelli1995uniform,chodosh2020minimal}). More precisely, solutions of \eqref{eq: intro, AC} resemble minimal surfaces in the sense that, as $\varepsilon\rightarrow 0$, the level surfaces of $u$ converge to a minimal surface in a suitable sense. 

From the foundational $L_\text{loc}^1$ convergence result of Modica \cite{modica1987gradient}, numerous notions of convergence have been studied. For example, in seminal paper of Hutchinson--Tonegawa \cite{hutchinson2000convergence}, they constructed the convergence result in varifold setting. Then, in 2006, Caffarelli--Cordoba proved uniform $C^{1,\alpha}$ convergence of the interfaces for all $\delta\in [0,2]$, under the assumption that each level surfaces are locally uniformly Lipschitz graphs. In 2019, Wang and Wei \cite{wang2019finite,wang2015structure} extended this to uniform $C^{2,\alpha}$ result for $\delta=2$, with the sheet separation result in $n=2$, and this result was further extended to manifold setting and $n=3$ by Chodosh--Mantoulidis \cite{chodosh2020minimal}. For the free-boundary version $\delta=0$, this uniform $C^{2,\alpha}$ result in general manifold setting with arbitrary dimension was followed by the first author \cite{an2025second}. 

\subsection{Free boundary Allen--Cahn}
Similar to the classical $\delta=2$ case, one can expect similar phase transition phenomena for $\delta\in [0,2)$ case, which leads to free boundary problems, already introduced by Caffarelli and Cordoba \cite{caffarelli1995uniform}. They were studied, for example, in \cite{du2022four,kamburov2013free,liu2018free,valdinoci2004plane,valdinoci2006flatness,wang2015structure}. Most interesting case in this free-boundary generating Allen--Cahn equations is the other extremity, i.e. $\delta=0$. In this setting, we obtain the indicator potential $\chi_{(-1,1)}(u)$. By perturbing the energy functional $J_\varepsilon^0(u;\Omega)$, any critical point $u$ of $J_\varepsilon^0$ satisfies the following free boundary problem in the viscosity sense (see \cite{de2009existence}): 
\begin{equation}\label{eq:FBAC}
    \left\{
        \begin{alignedat}{2}
            \Delta u&=0&\quad&\text{in}\quad\{|u|<1\}\\
            |\nabla u|&=1/\varepsilon&\quad&\text{on}\quad\partial\{|u|<1\}.
        \end{alignedat}
    \right.
\end{equation}
To provide a geometric illustration of the problem, consider a band with a width comparable to $2\varepsilon$. This band is composed of transition layers of $u$, i.e., $\{|u|<1\}$, where $u$ is harmonic inside. 

Mathematically speaking, the free boundary Allen--Cahn equation \eqref{eq:FBAC} enjoys a simpler structure than that of the classical Allen--Cahn equation. The ambient function $u$ is harmonic in the transition layers, and all the nonlinearities are ``concentrated'' as the free boundary condition $|\nabla u|=1/\varepsilon$. It also enjoys the property that separate interfaces (that is to say, the interfaces that are separated by free boundaries) do not interact with each other, unlike the classical Allen--Cahn equation. Therefore, a free boundary Allen--Cahn model can also serve as a natural framework for approximating minimal surfaces, that may be more tractable, in some cases, than the classical Allen--Cahn equation.
 
Due to its structural simplicity, recently the free boundary Allen--Cahn equation has attracted considerable attention, and produced new or stronger results, that are still open for the classical Allen--Cahn equation counterpart. Recently, in \cite{an2025second}, the first author showed that the transition layers are uniformly $C^{2,\alpha}$, given they are locally Lipschitz graphs. In classical Allen--Cahn setting, this uniform $C^{2,\alpha}$ regularity requires finite Morse index or stability assumptions and dimensional restriction $n\leq 10$ (\cite{wang2019second}), due to the presence of interaction between different sheets. As noted above, the free boundary Allen--Cahn equation lacks this sheet-interaction, thus it could provide uniform $C^{2,\alpha}$ estimate in full generality, without stability assumption nor dimensional restriction. At the same time, a monumental achievement came for the long-standing (local) De Giorgi conjecture, from Chan, Fernandez-Real, Figalli, Serra \cite{chan2025global}. They proved the free boundary version stable De Giorgi conjecture for $n=3$, and as a corollary, the free boundary version monotone De Giorgi conjecture for $n=4$. This was the first time a local stable De Giorgi type conjecture was resolved in this higher dimension.

\subsection{Hutchinson-Tonegawa theory}

Apart from its purely intriguing properties, the Allen--Cahn equation showed a possibility to be a powerful tool to attack problems in minimal surface theory. One famous application of such is by Gaspar--Guaraco and Guaraco \cite{gaspar2018allen,guaraco2018min}, where they adapts the PDE approach to construct minimal surfaces, and as a result, provide an alternative proof of famous Yau's conjecture, in comparison to the Almgren--Pitts min-max construction \cite{irie2018density,song2023existence}. In this framework, minimal hypersurfaces arise as sharp-interface limits of solutions to a semilinear elliptic equation, with the Allen--Cahn energy playing the role of a diffuse area functional. Moreover, Chodosh--Mantoulidis \cite{chodosh2020minimal} used the limit interface of the Allen--Cahn equation to resolve multiplicity one conjecture for a special case, while the general answer came later by Zhou \cite{zhou2020multiplicity}.

As highlighted above, we may assume that for some applications, the free boundary Allen--Cahn equation can deliver simpler approach than that of the classical Allen--Cahn equation. However, in the free boundary Allen--Cahn equation, we do not yet have celebrated Hutchinson--Tonegawa theory \cite{hutchinson2000convergence}, which shows the solutions of the Allen--Cahn equation converge to a minimal surface in the varifold sense. Moreover, Hutchinson--Tonegawa theory does not immediately generalize to the free boundary case, as the original proofs are dependent on elliptic estimates with nonlinearity, whereas in the free boundary Allen--Cahn equation we have distribution-type nonlinearity, so standard theory does not apply. This varifold convergence is often needed for applications; in particular, it was an essential tool in two important applications mentioned above \cite{gaspar2018allen,guaraco2018min,chodosh2020minimal}.

In this paper, we aim to provide a foundational varifold convergence result of Hutchinson-Tonegawa, in the free boundary Allen--Cahn analogue. Precisely, we have the next theorem:  

\begin{theorem}
\label{thm: main-euclidean}
	Let $\Omega\subset\mathbb{R}^n$ be a smooth connected domain, and $ V_i $ be varifolds associated with solutions $ u_{ \varepsilon_i } $ of the free boundary Allen--Cahn equation as in Definition \ref{def: classical_sol}. Moreover, assume Assumption \ref{assm: energy}. Then, taking a subsequence if necessary, we have
	\[
		u_{ \varepsilon } \to u_0 \in BV ( \Omega ; \{ \pm 1 \} ) \text{ for a.e.}, \quad V_i \to V \text{ in the varifold sense, }
	\]
	and $ V $ is an ($ n - 1 $)-rectifiable varifold. Moreover, we have
	\begin{enumerate}
		\item For each $ \phi \in C_c ( \Omega ) $,
		\[
			\lV V \rV ( \phi ) = \lim_{ i \to \infty } \int_{ \{ \lv u_{ \varepsilon_i } \rv < 1 \} \cap \Omega } \phi \bigg( \varepsilon_i \lv \nabla u_{ \varepsilon_i } \rv^2 + \frac{ \chi_{ ( - 1  , 1 ) } ( u_{ \varepsilon_i } ) }{ \varepsilon_i } \bigg),
		\]
		and $ V $ is stationary, that is, $ \delta V = 0 $.
		\item $ \spt \lv \partial \{ u_0 = 1 \} \rv \subset \spt \lV V \rV $, and $ u_{ \varepsilon_i } $ converges locally uniformly to $ \pm 1 $ on $ \Omega \setminus \spt \lV V \rV $.
		\item For each $ \tilde{ \Omega } \subset \subset \Omega $, $ \{ \lv u_{ \varepsilon_i } \rv < 1 \} \cap \tilde{ \Omega } $ converges to $ \spt \lV V \rV \cap \tilde{ \Omega } $ in the Hausdorff distance sense.
        \item Furthermore, if we assume $\|D^2u_\varepsilon\|_{L^\infty}\leq C\varepsilon^{-2}$ is uniformly bounded in terms of $\varepsilon$, then the limit varifold $V$ is integral, and the density $ \theta = 4 N $ of $ V $ satisfies
        \[
            N ( x ) = \begin{cases}
                \text{ odd } &\mathcal{ H }^{ n - 1 }\text{-}a.e.\,x \in \partial^* \{ u_0 = 1 \},\\
                \text{ even } &\mathcal{ H }^{ n - 1 }\text{-}a.e.\,x \in \Omega \setminus \partial^* \{ u_0 = 1 \}.
            \end{cases}
        \]
	\end{enumerate}
\end{theorem}

\begin{remark}
    Notice that unlike the classical Allen--Cahn case, we require the uniform $C^2$ bound $\|D^2u_\varepsilon\|_{L^\infty}\leq C\varepsilon^{-2}$, to have the integrality of the limit varifold $V$. This assumption is essential due to the fact that there are no known $C^2$ regularity results on the solutions (without stability or finite Morse index) of the free boundary Allen--Cahn equation. A priori, complicated free boundary behavior can occur, and without having a uniform $C^2$ estimate, we may have concentrated energy around the free boundary, which hinders the proof of the integrality of the limit varifold $V$. This is in sharp contrast to the classical Allen--Cahn equation, where one can immediately prove such a uniform $C^2$ estimate using standard elliptic estimates, and no such ``concentration of energy at the boundary'' can occur. Proving such a uniform $C^2$ estimate for the free boundary Allen--Cahn equation will be an interesting problem.
\end{remark}

Surprisingly, the proof of Theorem \ref{thm: main-euclidean} is not just a straightforward adjustment from the classical Allen--Cahn equation \cite{hutchinson2000convergence}, due to the presence of the free boundary, where the solution $u$ fails to be a $C^1$ function. Due to this presence of the free boundary, we cannot directly use the standard toolkit in elliptic PDEs, and this difficulty is highlighted in the proof of the integrality of the limit varifold (see Section \ref{seq: Integrality}). Therefore, we present novel arguments that rely on the properties of free boundary equations in the context of the Bernoulli problem.

In addition to the varifold convergence, we also provide $\Gamma$-convergence of the free boundary Allen--Cahn energy functional, which originally given by Modica and Mortola \cite{mortola1977esempio}:

\begin{theorem}[$\Gamma$-convergence]
\label{thm: Gamma convergence}
    For $ J_\varepsilon $, we have the following:
    \begin{enumerate}
        \item For all $ \{ u_{ \varepsilon } \}_{ \varepsilon > 0 } \subset L^1 ( \Omega ; [ - 1 , 1 ] ) $ and $ u \in L^1 ( \Omega ; [ - 1 , 1 ] ) $ such that $u_\varepsilon\xrightarrow{L^1(\Omega)}u$, we have
        \[
            J_0 ( u ) \leq \liminf_{ \varepsilon \to + 0 } J_\varepsilon ( u_{ \varepsilon } ).
        \]
        \item For all $ u \in L^1 ( \Omega ; [ - 1 , 1 ] ) $, there exists a sequence $ \{ u_{ \varepsilon } \}_{ \varepsilon > 0 } \subset L^1 ( \Omega ; [ - 1 , 1 ] ) $ such that it satisfies $u_\varepsilon\xrightarrow{L^1(\Omega)}u$, and
        \[
            J_0 ( u ) \geq \limsup_{ \varepsilon \to + 0 } J_\varepsilon ( u_{ \varepsilon } ).
        \]
    \end{enumerate}
\end{theorem}

Provided with the $\Gamma$-convergence, we show that local minimization property is preserved under the limit $\varepsilon\rightarrow 0$. This implies that the limiting varifold is locally minimizing, and thus smooth apart from a set of measure with Hausdorff dimension $n-8$ (see \cite{simon1983lectures}).

\begin{theorem}
\label{thm: min to min}
	Let $ \{ u_{ \varepsilon_i } \}_{ i = 1 }^{ \infty } \subset W^{ 1 , 2 } ( \Omega ; [ - 1 , 1 ] ) $ with $ \varepsilon_i \to 0 $ as $ i \to \infty $ be such that
	\begin{enumerate}
		\item satisfying Assumption \ref{assm: energy}, i.e. $ J_{\varepsilon_i} ( u_{ \varepsilon_i } ) \leq E_0 < \infty $ for all $ i $,
		\item there exists $ c > 0 $ such that $ J_{\varepsilon_i} ( u_{ \varepsilon_i } ) \leq J_{ \varepsilon_i } ( \tilde{ u } ) $ for all $ \tilde{ u } \in W^{ 1 , 2 } ( \Omega ; [ - 1 , 1 ] ) $ with $ \int_\Omega \lv u_{ \varepsilon_i } - \tilde{ u } \rv < c $.
	\end{enumerate}
	Let $ u_0 $ be as in Theorem \ref{thm: main-euclidean}. Then, $ u_0 $ satisfies that, for any $ \tilde{ u } \in BV ( \Omega ; \{ \pm 1 \} ) $ with $ \int_\Omega \lv u_0 - \tilde{ u } \rv < c $, $ J_0 ( u_0 ) \leq J_0 ( \tilde{ u } ) $ and $ \spt \lV V \rV = \spt \lv \partial \{ u_0 = 1 \} \rv $ hold. Additionally, if we have the $C^2$ uniform estimate $\|D^2u_{\varepsilon_i}\|_{L^\infty}\leq C\varepsilon_i^{-2}$, then $ 4^{ - 1 } \lV V \rV = \lv \partial \{ u_0 = 1 \} \rv $.
\end{theorem}

With this Hutchinson--Tonegawa theory provided for the free boundary Allen--Cahn equation, we expect to be able to use free boundary Allen--Cahn equation to re-prove famous applications such as Gaspar--Guaraco and Guaraco \cite{gaspar2018allen,guaraco2018min} or Chodosh--Mantoulidis \cite{chodosh2020minimal}, possibly with much shorter arguments. These will remain as our future projects.

\subsection{Organization of the paper}

In Section \ref{sec: preliminaries}, we provide the definition and essential assumptions of the solution $u$ that we will use throughout the paper. Moreover, we provide the notion of varifold and show that one can associate a varifold to a solution of the Allen--Cahn equation. In Section \ref{sec: monotonicity}, we show the monotonicity formula of Hutchinson--Tonegawa theory. In Section \ref{sec: rec}, we provide the rectifiability of the limit varifold. In Section \ref{seq: Integrality}, we prove the integrality of the limit varifold. Finally, in Section \ref{sec: proofs of main}, we collect the results and complete the proofs of the main theorems.

\subsection*{Notations}

We write here a list of symbols used throughout the paper.

\begin{center}
\begin{tabular}{p{1.5cm}p{12cm}}
$\Omega$ & Open bounded and connected domain with smooth boundary in $\mathbb{R}^n$\\
$u_\varepsilon$ & Function as in Theorem \ref{thm: main-euclidean} \\
$J_\varepsilon$ & Energy defined by \eqref{eq: FBAC}\\
$B_r$ & $n$-dimensional ball with radius $r>0$ \\
$\nu$ & $\nabla u/|\nabla u|$, unit normal vector of level surfaces of $u_\varepsilon$\\
$V_i$ & Varifold associated to $u_{\varepsilon_i}$, see \eqref{eq: def of V i} \\
$V$ & Limit varifold as $\varepsilon_i\rightarrow 0$, see \eqref{eq: def of lim varifold}\\
$\mu_i$ & Radon measure associated to the energy, see \eqref{eq: def of mu i}\\
$\mu$ & Limit measure as $\varepsilon_i\rightarrow 0$, see \eqref{eq: def of lim Radon measure}\\
$e_i$ & Energy function of $u_{\varepsilon_i}$, see \eqref{eq: def of energy and discrepancy}\\
$\xi_i$ & Discrepancy of $u_{\varepsilon_i}$, see \eqref{eq: def of energy and discrepancy}\\
$\omega_n$ & Volume of $n$-dimensional unit ball in $\mathbb{R}^n$
\end{tabular}
\end{center}

\section{Preliminaries}
\label{sec: preliminaries}

Consider a smooth domain $\Omega\subset\mathbb{R}^n$. The free boundary Allen--Cahn energy $J_\varepsilon:=J_\varepsilon^0$ is defined by taking indicator potential in \eqref{eq:energy-functional}, i.e. $\delta=0$:
\begin{equation}
	\label{eq: FBAC}
	J_{ \varepsilon } ( u ) := \int_{ \Omega } \varepsilon \lv \nabla u \rv^2 + \frac{ \chi_{ ( - 1 , 1 ) } ( u ) }{ \varepsilon }.
\end{equation}

\subsection{Solution of Free Boundary Allen--Cahn Problem}
\begin{definition}
\label{def: stationary}
We call $ u : \Omega \to [ - 1 , 1 ] $ {\itshape stationary} of $ J_{ \varepsilon } $ if $ u $ satisfies
\begin{equation}
\label{eq: stationary}
	\delta J_{ \varepsilon } ( u ) [ g ] = 0
\end{equation}
for any smooth compactly supported vector field $g\in C_c^1(\Omega;\mathbb{R}^n)$, where the first variation $ \delta J_{ \varepsilon } ( u ) $ is defined by a linear functional
\begin{equation*}
		\delta J_{ \varepsilon } ( u ) [ g ] := \int_{ \{ \lv u \rv < 1 \} \cap \Omega } \bigg( - 2 \varepsilon \nabla u \cdot D g \nabla u + \varepsilon \lv \nabla u \rv^2 \dv g + \frac{ 1 }{ \varepsilon } \dv g \bigg).
\end{equation*}
\end{definition}

\begin{definition}
\label{def: classical_sol}
For $\Omega\subset\mathbb{R}^n$, we call $ u : \Omega \to [ - 1 , 1 ] $ {\itshape a classical solution} of $ J_{ \varepsilon } $ if $ u $ satisfies that
\begin{equation}
	\label{eq: Bernoulli_AC}
    \left\{
	\begin{alignedat}{3}
	    \Delta u &= 0 \quad&&\text{in} \quad&&\Omega \cap \{ \lv u \rv < 1 \},\\
		\lv \nabla u \rv &= 1 / \varepsilon \quad&&\text{on}\quad &&\Omega \cap \partial \{ \lv u \rv < 1 \},
	\end{alignedat}
    \right.
\end{equation}
and $ \partial \{ \lv u \rv < 1 \} $ is a locally $ C^1 $ surface. 
\end{definition}
\begin{remark}
\label{rem: conseqences from def}
    Here, we state a couple of simple facts concerning the two definitions above. By direct calculation using integral by parts, one can see that any classical solution is, in particular, stationary. Moreover, by the fact that a harmonic function in a domain with Lipschitz boundary is (locally) Lipschitz (see \cite[Section 11]{caffarelli2005geometric}, for example), we conclude that a classical solution is (locally) Lipschitz. We will use these facts in Section \ref{sec: monotonicity}.
\end{remark}
    
\subsection{Assumption and easy consequence}

Note that we do not assume any energy minimality nor the stability. Instead, in this paper, we always assume the next energy bound:
\begin{assumption}
	\label{assm: energy}
	Let $ \{ u_{ \varepsilon_i } \}_{ i = 1 }^{ \infty } $ be a sequence of $ W^{ 1 , 2 }(\Omega)$ functions that satisfies Definition \ref{def: classical_sol} for $ \varepsilon_i \in ( 0 , 1 ) $. Here, $ \lim_{ i \to \infty } \varepsilon_i = 0 $. Throughout the paper, we assume that there exists $ 0 < E_0 < \infty $ such that
	\[
		J_{ \varepsilon_i } ( u_{ \varepsilon_i } ) \leq E_0, \quad \text{for all } i.
	\]
\end{assumption}

We discuss a few immediate consequences from Assumption \ref{assm: energy}.   By the Cauchy--Schwarz inequality, we have
\begin{equation}
\label{eq: AMGM for J}
    \int_{ \Omega } \lv \nabla u_{ \varepsilon_i } \rv \leq \frac{ 1 }{ 2 } \int_{ \{ \lv u_{ \varepsilon_i } \rv < 1 \} \cap \Omega } \varepsilon_i \lv \nabla u_{ \varepsilon_i } \rv^2 + \frac{ \chi_{ ( - 1 , 1 ) } ( u_{ \varepsilon_i } ) }{ \varepsilon } \leq \frac{ E_0 }{ 2 }.
\end{equation}
By the compactness theorem for bounded variation functions, there exists a subsequence (by abbreviation we denote by $ \{ u_{ \varepsilon_i } \} $) and an almost everywhere pointwise limit $ u_0 $ such that
\[
	\lim_{ i \to \infty } \int_{ \Omega } \lv u_{ \varepsilon_i } - u_0 \rv = 0 \text{ and } \int_{ \Omega } \lv \nabla u_0 \rv \leq \liminf_{ i \to \infty } \int_{ \Omega } \lv \nabla u_{ \varepsilon_i } \rv,
\]
where $ \lv \nabla u_0 \rv $ is the total variation of the vector-valued Radon measure $ \nabla u_0 $. Moreover, by Fatou's Lemma and the energy bound, we have
\[
	\int_{ \Omega } \chi_{ ( - 1 , 1 ) } ( u_0 ) \leq \int_{ \Omega } \lim_{ i \to \infty } \chi_{ ( - 1 , 1 ) } ( u_{ \varepsilon_i } ) \leq \liminf_{ i \to \infty } \int_{ \Omega } \chi_{ ( - 1  , 1 ) } ( u_{ \varepsilon_i } ) \leq \liminf_{ i \to \infty } ( \varepsilon_i E_0 ) = 0,
\]
which implies that $ u_0 = \pm 1 $ for almost everywhere in $ \Omega $. In summary, we have the following proposition.

\begin{prop}
\label{prop: BV convergence of u_i}
	Let a family $ \{ u_{ \varepsilon_i } \}_{ i = 1 }^{ \infty } $ be as in Assumption \ref{assm: energy}. Taking a subsequence if necessary, we have
	\[
		u_{ \varepsilon_i } \to u_0 \in BV ( \Omega ; \{ \pm 1 \} ) \text{ a.e. and } \int_{ \Omega } \lv \nabla u_0 \rv \leq \frac{ E_0 }{ 2 }.
	\]
\end{prop}

\subsection{Associated Varifold}
\label{subseq: varifold}

In this subsection, we recall the notions of the varifold and associate to solutions of the free boundary Allen--Cahn problem in a varifold in a natural way. We refer to \cite{allard1972first,simon1983lectures} for the detailed explanation on the varifold.

Let $ G ( n , n - 1 ) $ denote the Grassmanian manifold of unoriented $ ( n - 1 ) $-dimensional subplanes in $ \R^n $. With the abuse of the notation, we write $ S \in G ( n , n - 1 ) $ as the orthogonal projection of $ \R^n $ onto $ S $, and $ S_1 \cdot S_2 = \tr ( {}^t S_1 \circ S_2 ) $ for $ S_1 , S_2 \in G ( n , n - 1 ) $. We say that $ V $ is an $ ( n - 1 ) $-dimensional varifold in $ \Omega \subset \R^n $ if $ V $ is a Radon measure on $ \Omega \times G ( n , n - 1 ) $. Convergence in the varifold sense means convergence in the usual sense of Radon measure. Let $ \V_{ n - 1 } ( \Omega ) $ denote the set of all $ ( n - 1 ) $-dimensional varifolds in $ \Omega $. For $ V \in \V_{ n - 1 } ( \Omega ) $, let the weight $ \lV V \rV $ be the Radon measure in $ \Omega $ defined by
\[
	\lV V \rV ( A ) := V ( \{ ( x , S ) \mid x \in A , S \in G ( n , n - 1 ) \} )
\]
for each Borel set $ A \subset \Omega $. We call $ V \in \V_{ n - 1 } ( \Omega ) $ rectifiable if there exist an $ \mathcal{ H }^{ n - 1 } $-measurable countably ($  n - 1 $)-rectifiable set $ M \subset \Omega $ (see \cite{simon1983lectures} for the definition and its properties) and a locally $ \mathcal{ H }^{ n - 1 } $-integrable function $ \theta $ defined on $ M $ such that
\[
	V ( \phi ) = \int_M \phi ( x , T_x M ) \theta ( x ) \, d \mathcal{ H }^{ n - 1 } ( x ) \text{ for } \phi \in C_c ( \Omega \times G ( n , n - 1 ) ).
\]
Here $ T_x M $ is the approximate tangent space of $ M $ at $ x $ which exists $ \mathcal{ H }^n $-almost everywhere on $ M $ and we say $ \theta $ the density of $ V $. If $ \theta \in \N $ for $ \mathcal{ H }^{ n - 1 } $-almost everywhere on $ M $, we say $ V $ is integral. We define the first variation of $ V $ by
\[
	\delta V ( g ) := \int_{ \Omega \times G ( n , n - 1 ) } D g \cdot S \, d V ( x , S )
\]
for any $ g \in C^1_c ( \Omega ; \R^n ) $. We define the total variation $ \lV \delta V \rV ( U ) $ by
\[
	\lV \delta V \rV ( U ) := \sup \{ \delta V ( g ) \mid g \in C^1_c ( U ; \R^n ) , \lv g \rv \leq 1 \}
\]
wherever $ U \subset \Omega $ is open. If $ \lV \delta V \rV $ is locally finite, one can regard it as a Radon measure.

We associate to each solution $ u_{ \varepsilon_i } $ a varifold $ V_i $ as follows: First, we define a functional $ \mu_i : C_c ( \Omega ; [ 0 , \infty ) ) \to \R $ by
\begin{equation}\label{eq: def of mu i}
    \mu_i ( \phi ) := \int_{ \Omega } \phi \Bigg( \varepsilon \lv \nabla u_{ \varepsilon_i } \rv^2 + \frac{ \chi_{ ( - 1 , 1 ) } ( u_{ \varepsilon_i } ) }{ \varepsilon_i } \Bigg) \text{ for } \phi \in C_c ( \Omega ; [ 0 , \infty ) ).
\end{equation}
Note that these functionals $ \mu^i $ are Radon measures on $ \Omega $ by the Riesz representation theorem. Define $ V_i \in \V_{ n - 1 } ( \Omega ) $ by
\begin{equation}\label{eq: def of V i}
    	V_i ( \phi ) := \int_{ \{ \lv \nabla u_{ \varepsilon_i } \rv \neq 0 \} } \phi \bigg( x , I - \frac{ \nabla u_{ \varepsilon_i } ( x ) }{ \lv \nabla u_{ \varepsilon_i } ( x ) \rv } \otimes \frac{ \nabla u_{ \varepsilon_i } ( x ) }{ \lv \nabla u_{ \varepsilon_i } ( x ) \rv } \bigg) \, d \mu_i ( x )
\end{equation}
for $ \phi \in C_c ( \Omega \times G ( n , n - 1 ) ) $, where $ I $ denotes $ ( n \times n ) $-identity matrix and $ \otimes $ is the tensor product of the two vectors. By definition, we have
\[
	\lV V_i \rV = \mu_i \lfloor_{ \{ \lv \nabla u_{ \varepsilon_i } \rv \neq 0 \} } 
\]
and
\[
	\delta V_i ( g ) = \int_{ \{ \lv \nabla u_{ \varepsilon_i } \rv \neq 0 \} } D g \cdot \bigg( I - \frac{ \nabla u_{ \varepsilon_i } }{ \lv \nabla u_{ \varepsilon_i } \rv } \otimes \frac{ \nabla u_{ \varepsilon_i } }{ \lv \nabla u_{ \varepsilon_i } \rv } \bigg) \, d \mu_i
\]
for each $ g \in C^1_c ( \Omega ; \R^n ) $.

\section{Local Monotonicity Formula}
\label{sec: monotonicity}
In the case of the classical Allen--Cahn equations, Modica's inequality \cite{modica1985gradient} plays a key role to study phase transition layers. An analogous inequality in the free boundary setting holds true as well (e.g., \cite[Lemma 10.4]{chan2025global}). For the completeness, we provide the proof here.

\begin{lemma}
\label{lem: modica}
Let $ 0 < \varepsilon < 1 $ and let $ u_{ \varepsilon } $ be a classical solution of $ J_{ \varepsilon } $ in $ \Omega $. We then have Modica type pointwise estimate
\begin{equation}
\label{eq: modica}
	\varepsilon \lv \nabla u_{ \varepsilon } \rv^2 \leq \frac{ 1 }{ \varepsilon }.
\end{equation}
\end{lemma}
\begin{proof}
    Let $ B_{ 2 r } ( x ) \subset \Omega $ arbitrarily. For this lemma, it is sufficient to prove that
	 \[
	 	\varepsilon \lv \nabla u \rv^2 \leq \frac{ 1 }{ \varepsilon } \text{ in } B_r ( x ).
	 \]
	 Without loss of generality, we can choose $ r = 1 $ and $ x = 0 $. Moreover, by rescaling $ x \mapsto x / \varepsilon $, we consider the rescaled Bernoulli problem
	 \begin{equation}
	 \label{eq: rescaled problem}
	 	\begin{cases}
	 		\Delta u = 0 &(\text{in } B_{ 2 \varepsilon^{ - 1 } } ( 0 ) \cap \{ \lv u \rv < 1 \} ),\\
	 		\lv \nabla u \rv = 1 &(\text{on } B_{ 2 \varepsilon^{ - 1 } } ( 0 ) \cap \partial \{ \lv u \rv < 1 \} ),
	 	\end{cases}
	 \end{equation}
	 and then the desired inequality becomes
	 \[
	 	\lv \nabla u \rv^2 \leq 1 \text{ in } B_{ \varepsilon^{ - 1 } } ( 0 ).
	 \]
	 For a contradiction, we assume that $ c := \sup_{ B_{ \varepsilon^{ - 1 } } ( 0 ) } ( \lv \nabla u \rv^2 - 1 ) > 0 $. Note that the constant $ c $ is finite because $ u $ is a Lipshitz function (Remark \ref{rem: conseqences from def}). Since $ \lv \nabla u \rv = 1 $ on the free boundary of $ u $, the maximum is attained at an interior point in $ \{ \lv u \rv < 1 \} $. Let $ \zeta \in C^{ \infty }_c ( B_{ 2 \varepsilon^{ - 1 } - 1 } ( 0 ) ) $ be such that
	 \[
	 	\zeta \equiv 1 \text{ in } B_{ \varepsilon^{ - 1 } } ( 0 ),\ \lv \nabla \zeta \rv \leq 2 \varepsilon,\ \lv \Delta \zeta \rv \leq 2 \varepsilon^2,\ 0 \leq \zeta \leq 1 \text{ in } B_{ 2 \varepsilon^{ - 1 } - 1 } ( 0 ).
	 \]
	 Let us consider $ \tilde{ \xi } := \lv \nabla u \rv^2 - 1 - c \varepsilon ( 1 - u^2 ) + c \zeta $. By \eqref{eq: rescaled problem}, we have
	 \[
	 	\tilde{ \xi } \leq c \text{ on } \partial ( B_{ 2 \varepsilon^{ - 1 } - 1 } ( 0 ) \cap \{ \lv u \rv < 1 \} ) \text{ and } \sup_{ B_{ 2 \varepsilon^{ - 1 } - 1 } ( 0 ) \cap \{ \lv u \rv < 1 \} } \tilde{ \xi } \geq ( 2 - \varepsilon ) c.
	 \]
	 Therefore, there is an interior maximum point $ x_0 \in B_{ 2 \varepsilon^{ - 1 } - 1 } ( 0 ) \cap \{ \lv u \rv < 1 \} $ of $ \tilde{ \xi } $ and we have the following properties:
	 \[
	 	\lv \nabla u ( x_0 ) \rv^2 \geq 1, \quad \Delta \tilde{ \xi } ( x_0 ) \leq 0.
	 \]
	 From this and \eqref{eq: rescaled problem}, it follows that
	 \[
	 	\Delta ( \lv \nabla u \rv^2 - 1 - c \varepsilon ( 1 - u^2 ) ) ( x_0 ) \leq - c \Delta \zeta ( x_0 ) \leq 2 c \varepsilon^2
	 \]
	 and
	\[
	 		\Delta ( \lv \nabla u \rv^2 - 1 - c \varepsilon ( 1 - u^2 ) ) ( x_0 ) = 2 \lv \nabla^2 u ( x_0 ) \rv^2 + 2 c \varepsilon \lv \nabla u ( x_0 ) \rv^2 \geq 2 c \varepsilon,
	\]
	which is a contradiction. This completes the proof.
\end{proof}

As a consequence of the above Modica inequality, we obtain the following monotonicity formula. 

\begin{lemma}
\label{lem: monotonicity}
Let $ u_{ \varepsilon } $ be a classical solution of $ J_{ \varepsilon } $ in $ \Omega $. For any $ B_r ( x ) \subset \Omega $, we have
\begin{equation}
\label{eq: monotonicity}
\begin{split}
	&\frac{ d }{ d r } \Bigg( \frac{ 1 }{ r^{ n - 1 } } \int_{ B_r ( x ) \cap \{ \lv u_{ \varepsilon } \rv < 1 \} } \bigg( \varepsilon \lv \nabla u_{ \varepsilon } \rv^2 + \frac{ \chi_{ ( - 1 , 1 ) } ( u_{ \varepsilon } ) }{ \varepsilon } \bigg) \Bigg)\\
	&= \frac{ 1 }{ r^n } \int_{ B_r ( x ) \cap \{ \lv u_{ \varepsilon } \rv < 1 \} } \bigg( \frac{ \chi_{ ( - 1 , 1 ) } ( u_{ \varepsilon } ) }{ \varepsilon } - \varepsilon \lv \nabla u_{ \varepsilon } \rv^2 \bigg) + \frac{ 2 \varepsilon }{ r^{ n + 1 } } \int_{ \partial B_r ( x ) \cap \{ \lv u_{ \varepsilon } \rv < 1 \} } \big( ( y - x ) \cdot \nabla u_{ \varepsilon } \big)^2
\end{split}
\end{equation}
in the distributional sense. In particular, it follows from \eqref{eq: modica} that
\[
	\frac{ 1 }{ r^{ n - 1 } } \int_{ B_r ( x ) \cap \{ \lv u_{ \varepsilon } \rv < 1 \} } \varepsilon \lv \nabla u_{ \varepsilon } \rv^2 + \frac{ \chi_{ ( - 1 , 1 ) } ( u_{ \varepsilon } ) }{ \varepsilon }
\]
is non-decreasing with respect to $ r $.
\end{lemma}
\begin{proof}
	By a suitable translation, we let $ x = 0 $ and let $ g^j ( y ) = y^j \rho ( \lv y \rv ) $, where $ \rho ( \lv y \rv ) $ is a smooth approximation to the characteristic function $ \chi_{ B_r ( 0 ) } $. Note that $ u_{ \varepsilon } $ is, in particular, stationary, that is, $ u_{ \varepsilon } $ satisfies \eqref{eq: stationary} (Remark \ref{rem: conseqences from def}). Plugging this $ g = ( g^1 , \dots , g^n ) $ in \eqref{eq: stationary}, we have
	\begin{equation*}
		\int_{ \{ \lv u_{ \varepsilon } \rv < 1 \} \cap \Omega } \Bigg( \bigg( \varepsilon \lv \nabla u_{ \varepsilon } \rv^2 + \frac{ 1 }{ \varepsilon
		 } \bigg) ( \lv y \rv \rho^{ \prime } + n \rho ) - 2 \varepsilon \frac{ \rho^{ \prime } }{ \lv y \rv } ( y \cdot \nabla u_{ \varepsilon } )^2 - 2 \varepsilon \lv \nabla u_{ \varepsilon } \rv^2 \rho \Bigg) = 0.
	\end{equation*}
	By letting $ \rho \to \chi_{ B_r ( 0 ) } $ and rearranging the terms, we obtain
	\begin{equation*}
		\begin{split}
			&- ( n - 1 ) \int_{ B_r ( 0 ) \cap \{ \lv u_{ \varepsilon } \rv < 1 \} } \bigg( \varepsilon \lv \nabla u_{ \varepsilon } \rv^2
			+ \frac{ 1 }{ \varepsilon
			} \bigg) + r \int_{ \partial B_r ( 0 ) \cap \{ \lv u_{ \varepsilon } \rv < 1 \} } \bigg( \varepsilon \lv \nabla u_{ \varepsilon } \rv^2 + \frac{ 1 }{ \varepsilon
			} \bigg)\\
			&= \int_{ B_r ( 0 ) \cap \{ \lv u_{ \varepsilon } \rv < 1 \} } \bigg( \frac{ 1 }{ \varepsilon } - \varepsilon \lv \nabla u_{ \varepsilon } \rv^2 \bigg)
			+ \frac{ 2 \varepsilon }{ r } \int_{ \partial B_r ( 0 ) \cap \{ \lv u_{ \varepsilon } \rv < 1 \} } ( y \cdot \nabla u_{ \varepsilon } ).
		\end{split}
	\end{equation*}
	Dividing the above by $ r^n $ leads to \eqref{eq: monotonicity}.
\end{proof}

\section{Rectifiability of the Limit Varifold}
\label{sec: rec}
From Assumption \ref{assm: energy} and the compactness of Radon measure, taking a subsequence if necessary, it follows that there is the limit measure $ \mu $ on $ \Omega $, defined by
\begin{equation}
\label{eq: def of lim Radon measure}
	\mu ( \phi ) = \lim_{ i \to \infty } \mu_i ( \phi ) = \lim_{ i \to \infty } \int_{ \Omega } \phi \Bigg( \varepsilon \lv \nabla u_{ \varepsilon_i } \rv^2 + \frac{ \chi_{ ( - 1 , 1 ) } ( u_{ \varepsilon_i } ) }{ \varepsilon_i } \Bigg),
\end{equation}
for compactly supported continuous test functions $\phi\geq 0$. The limit varifold $ V $ can be taken for $ V_i $ as well, defined by
\begin{equation}
\label{eq: def of lim varifold}
	V ( \phi ) = \lim_{ i \to \infty } V_i ( \phi ) = \lim_{ i \to \infty } \int_{ \Omega \times G ( n , n - 1 ) } \phi \bigg( x , I - \frac{ \nabla u_{ \varepsilon_i } ( x ) }{ \lv \nabla u_{ \varepsilon_i } ( x ) \rv } \otimes \frac{ \nabla u_{ \varepsilon_i } ( x ) }{ \lv \nabla u_{ \varepsilon_i } ( x ) \rv } \bigg) \, d \mu_i ( x )
\end{equation}
for $ \phi \in C_c ( \Omega \times G ( n , n - 1 ) ) $ with $ \phi \geq 0 $, and we note that $ \mu = \lV V \rV $. To prove the rectifiability of $ \mu $, we require a lower bound on the density and local boundedness of the first variation. To this end, we need to ensure that the limiting Radon measure is non-degenerate in a measure-theoretic sense and that the discrepancy function $ \lv \varepsilon \lv \nabla u \rv^2 - \chi_{ ( - 1  , 1 ) } ( u ) / \varepsilon \rv $ vanishes.

We quote the following weak non-degeneracy property of the one phase Bernoulli problem from \cite[Lemma 3.5]{chan2025global}, which is an important tool to provide the lower bound of the density. 

\begin{lemma}[Clean ball property]
\label{lem: nondegeneracy}
	There exists $ \delta = \delta ( n ) > 0 $ such that the following holds: Let $ \rho > 0 $, $ y \in B_R ( 0 ) $ with $ B_{ 2 \rho } ( y ) \subset B_R ( 0 ) $, and let $ u $ be a classical solution of the following one phase Bernoulli problem in $ B_{ 2 \rho } ( y ) $:
	\begin{equation}
	\label{eq: one phase Bernoulli}
	\begin{cases}
		\Delta u = 0 &(\text{in } B_{ 2 \rho } ( y ) \cap \{ u > 0 \} ),\\
		\lv \nabla u \rv = 1&(\text{on } B_{ 2 \rho } ( y ) \cap \partial \{ u > 0 \} ).
	\end{cases}
	\end{equation}
	Suppose that there is a connected component $ U $ of $ \{ u > 0 \} \cap B_{ 2 \rho } ( y ) $ such that $ \mathcal{ L }^n ( U \cap B_{ 2 \rho } ( y ) ) \leq \delta \rho^n $. Then, $ U \cap B_{ \rho } ( y ) = \emptyset $.
\end{lemma}

We next show the estimate on the density by using the weak non-degeneracy and the monotonicity formula. Then we show that the discrepancy vanishes in $L^1_\text{loc}$ sense.

\begin{prop}
\label{prop: density}
	For all $ x \in \spt \mu \cap \Omega $, we have
	\[
		0 < \liminf_{ r \to 0 } \frac{ \mu ( B_r ( x ) ) }{ r^{ n - 1 } } \leq \limsup_{ r \to 0 } \frac{ \mu ( B_r ( x ) ) }{ r^{ n - 1 } } < \infty.
	\]
\end{prop}
\begin{proof}
	The finiteness of the upper density immediately follows from Lemma \ref{lem: monotonicity} and Assumption \ref{assm: energy}. Indeed, setting $ r_0 := \dist ( x , \partial \Omega ) $ for $ x \in \spt \mu \cap \Omega $, we have
    \[
        \frac{ \mu ( B_r ( x ) ) }{ r^{ n - 1 } } \leq \liminf_{ i \to \infty } \frac{ \mu_i ( B_r ( x ) ) }{ r^{ n - 1 } } \leq \liminf_{ i \to \infty } \frac{ \mu_i ( B_{ r_0 } ( x ) ) }{ r_0^{ n - 1 } } \leq \frac{ E_0 }{ r_0^{ n - 1 } } < \infty
    \]
    for all $ 0 < r \leq r_0 $.

	To establish the lower bound, we let $ x \in \spt \mu \cap \Omega $ and $ r \in ( 0 , \dist ( x , \partial \Omega ) ) $. We start by claiming that there exists a sequence $ \{ x_i \}_{ i = 1 }^{ \infty } \subset \{ \lv u_{ \varepsilon_i } \rv < 1 \} $ such that $ x_i $ converges to $ x $. Suppose there exists $ s > 0 $ such that $ B_s ( x ) \cap \{ \lv u_{ \varepsilon_i } \rv < 1 \} = \emptyset $ for sufficiently large $ i $ and $ B_s ( x ) \subset \Omega $. Then
	\[
		\mu ( B_s ( x ) ) \leq \liminf_{ i \to \infty } \mu_i ( B_s ( x ) ) = \liminf_{ i \to \infty } \int_{ B_s ( x ) } \Bigg( \varepsilon_i \lv \nabla u_{ \varepsilon_i } \rv^2 + \frac{ \chi_{ ( - 1 , 1 ) } ( u_{ \varepsilon_i } ) }{ \varepsilon_i } \Bigg) = 0,
	\]
	which is a contradiction to the choice of $x$. Therefore, in the following, we fix $\{x_i\}_{i=1}^\infty\subset\{|u_{\varepsilon_i}|<1\}$ such that $x_i\rightarrow x$.
    
    The monotonicity lemma (Lemma \ref{lem: monotonicity}) shows that
	\begin{equation}\label{eq: density lower bound proof 1}
		\begin{split}
			\frac{ \mu ( B_r ( x ) ) }{ r^{ n - 1 } }
			&\geq \lim_{ i \to \infty } \frac{ 1 }{ r^{ n - 1 } } \int_{ B_{ r / 2 } ( x_i ) \cap \{ \lv u_{ \varepsilon_i } \rv < 1 \} } \Bigg( \varepsilon_i \lv \nabla u_{ \varepsilon_i } \rv^2 + \frac{ \chi_{ ( - 1 , 1 ) } ( u_{ \varepsilon_i } ) }{ \varepsilon_i } \Bigg)\\
			&\geq \lim_{ i \to \infty } \frac{ 1 }{ ( 2 \varepsilon_i )^{ n - 1 } } \int_{ B_{ \varepsilon_i } ( x_i ) \cap \{ \lv u_{ \varepsilon_i } \rv < 1 \} } \Bigg( \varepsilon_i \lv \nabla u_{ \varepsilon_i } \rv^2 + \frac{ \chi_{ ( - 1 , 1 ) } ( u_{ \varepsilon_i } ) }{ \varepsilon_i } \Bigg)\\
			&\geq \frac{ 1 }{ 2^{ n - 1 } } \lim_{ i \to \infty } \mathcal{ L }^n ( \{ \lv u_{ \varepsilon_i } ( \varepsilon_i ( \cdot ) + x_i ) \rv < 1 \} \cap B_1 ( 0 ) ).
		\end{split}
	\end{equation}
	We let $ \tilde{ u }_i ( y ) := ( u_{ \varepsilon_i } ( \varepsilon_i y + x_i ) + 1 ) / 2 $ for $ y \in B_1 ( 0 ) $, we have $ \tilde{ u }_i ( 0 ) \in ( 0 , 1 ) $ and
	\[
        \left\{
        \begin{alignedat}{2}
            \Delta \tilde{ u }_i &= 0 \quad&&\text{in}\quad\{ 0 < \tilde{ u }_i < 1 \} \\
            \lv \nabla \tilde{ u }_i \rv &= 1/2 \quad&&\text{on}\quad \partial \{ 0 < \tilde{ u }_i < 1 \}.
        \end{alignedat}
		\right.
	\]
	Note that $ \lV \nabla \tilde{ u }_i \rV_{ L^{ \infty } ( B_1 ( 0 ) ) } \leq 1/2 $. Suppose
	\[
		\lim_{ i \to \infty } \mathcal{ L }^n ( \{ \lv u_{ \varepsilon_i } ( \varepsilon_i ( \cdot ) + x_i ) \rv < 1 \} \cap B_1 ( 0 ) ) = \lim_{ i \to \infty } \mathcal{ L }^n ( \{ 0 < \tilde{ u }_i < 1 \} \cap B_1 ( 0 ) ) = 0.
	\]
	We show that this leads to a contradiction. Then from \eqref{eq: density lower bound proof 1}, the claim follows. Passing to a subsequence if necessary, the functions $ \tilde{ u }_i $ converge locally uniformly in $ B_1 $ to continuous $ \tilde{ u }_{ \infty } $. From the continuity of $ \tilde{ u }_{ \infty } $ and the hypothesis, it follows that $ \tilde{ u }_{ \infty } \equiv 0 $ or $ 1 $ in $ B_1 $. We assume $ \tilde{ u }_{ \infty } \equiv 0 $, without loss of generality. Then we have
	\[
		\lim_{ i \to \infty } \mathcal{ L }^n ( \{ 0 < \tilde{ u }_i \leq 1 \} \cap B_1 ( 0 ) ) = 0.
	\]
	By Proposition \ref{prop: interpolation L1 Lip}, for sufficiently large $ i $, we have
	\[
		\lV \tilde{ u }_i \rV_{ L^{ \infty } ( B_1 ( 0 ) ) } \leq C ( n ) \lV \tilde{ u }_i \rV_{ L^1 ( B_1 ( 0 ) ) }^{ 1 / ( n + 1 ) } \leq C ( n ) \mathcal{ L }^n ( \{ 0 < \tilde{ u }_i \leq 1 \} \cap B_1 ( 0 ) )^{ 1 / ( n + 1 ) } < \frac{ 1 }{ 2 } ,
	\]
	which implies that $ 0 \leq \tilde{ u }_i < 1 / 2 $ in $ B_1 ( 0 ) $. Therefore, $ \tilde{ u }_i $ is particularly a classical solution to the one phase Bernoulli problem \eqref{eq: one phase Bernoulli} in $ B_1 ( 0 ) $. From the clean ball property (Lemma \ref{lem: nondegeneracy}), it follows that $ \{ \tilde{ u }_i > 0 \} \cap B_{ 1 / 2 } ( 0 ) = \emptyset $ for sufficiently large $ i $ so that $ \mathcal{ L }^n ( \{ 0 < \tilde{ u }_i < 1 \} \cap B_1 ( 0 ) ) \leq \delta $, where $ \delta $ is as in Lemma \ref{lem: nondegeneracy}. This is a contradiction to the assumption that $ \tilde{ u }_i ( 0 ) \in ( 0 , 1 ) $. Thus, we have the strict lower bound
	\[
		\frac{ \mu ( B_r ( x ) ) }{ r^{ n - 1 } } \geq \lim_{ i \to \infty } \mathcal{ L }^n ( \{ \lv u_{ \varepsilon_i } ( \varepsilon_i ( \cdot ) + x_i ) \rv < 1 \} \cap B_1 ( 0 ) ) > 0,
	\]
    as desired. Since the lower bound in the above is independent of $ r $, this completes the proof.
\end{proof}

\begin{remark}
\label{rem: differece with HT}
     It is worth mentioning the difference from the proof of the classical Allen--Cahn case \cite[Proposition 4.2]{hutchinson2000convergence}, concerning the above density estimate. In the classical Allen--Cahn setting, the density estimate is obtained by the structure of the potential together with the fact that the transition layer concentrates on the set $ \{ \lv u_i \rv < \alpha \} $ for some $ \alpha \in ( 0 , 1 ) $ depending only on the potential. This concentration property is proved by the standard maximum principle for regular solutions. However, the same argument does not apply in our case, because $u$ fails to be even $C^1$ function at the free boundary. As an alternative, we employ the clean ball property from \cite{chan2025global}.
\end{remark}

\begin{corollary}
\label{cor: uniform}
	Either $ u_{ \varepsilon_i } \to 1 $ or $u_{\varepsilon_i}\to -1$, uniformly on each connected compact subset of $ \Omega \setminus \spt \lV V \rV $. In particular, $ \spt \lv \partial \{ u_0 = 1 \} \rv \subset \spt \lV V \rV $. The terms $ \varepsilon_i \lv \nabla u_{ \varepsilon_i } \rv^2 $ converges uniformly to $ 0 $ on compact subsets of $ \Omega \setminus \spt \lV V \rV $.
\end{corollary}
\begin{proof}
    This corollary follows from the argument using the clean ball property and Proposition \ref{prop: interpolation L1 Lip} for the previous proposition. Indeed, let $ U $ be an open set in a connected component of $ \Omega \setminus \spt \lV V \rV $ and take arbitrary compact set $ K \subset U $. We may assume that $ u_{ \varepsilon_i } \to - 1 $ in $ L^1 ( U ) $ and a.e.\,pointwise without loss of generality from Proposition \ref{prop: BV convergence of u_i}. For a contradiction, we suppose that there exists $ \delta_0 > 0 $ such that, for all $ N \in \N $, there exist $ i \geq N $ and $ x_i \in K $ such that $ \lv u_{ \varepsilon_i } ( x_i ) + 1 \rv > \delta_0 $. We let $ \tilde{ u }_i ( y ) := ( u_{ \varepsilon_i } ( \varepsilon_i y + x_i ) + 1 ) / 2 $ as in the proof of Proposition \ref{prop: density}. By the same argument as in the proof of Proposition \ref{prop: density}, we obtain $ \tilde{ u }_i ( 0 ) \in ( 0 , 1 ) $ and
	\[
		\lV \tilde{ u }_i \rV_{ L^{ \infty } ( B_1 ( 0 ) ) }  < \frac{ 1 }{ 2 },
	\]
	which implies that $ 0 \leq \tilde{ u }_i < 1 / 2 $ in $ B_1 ( 0 ) $. Therefore, $ \tilde{ u }_i $ is particularly a classical solution to the one phase Bernoulli problem \eqref{eq: one phase Bernoulli} in $ B_1 ( 0 ) $. From the clean ball property (Lemma \ref{lem: nondegeneracy}), it follows that $ \{ \tilde{ u }_i > 0 \} \cap B_{ 1 / 2 } ( 0 ) = \emptyset $ for sufficiently large $ i $ so that $ \mathcal{ L }^n ( \{ 0 < \tilde{ u }_i < 1 \} \cap B_1 ( 0 ) ) \leq \delta $, where $ \delta $ is as in Lemma \ref{lem: nondegeneracy}. This is a contradiction to the assumption that $ \tilde{ u }_i ( 0 ) \in ( 0 , 1 ) $.
\end{proof}

We then prove the discrepancy vanishes in the limit in $L_\text{loc}^1(\Omega)$.

\begin{prop}
\label{prop: dicrepancy vanishes}
	Let $ \xi_i := \varepsilon_i \lv \nabla u_{ \varepsilon_i } \rv^2 - \chi_{ ( - 1 , 1 ) } ( u_{ \varepsilon_i } ) / \varepsilon_i $. We then have $ \lv \xi_i \rv \to 0 $ in $ L^1_{ loc } ( \Omega ) $.
\end{prop}
\begin{proof}
	One can take the limit Radon measure $ \lv \xi \rv $ such that $ \int \lv \xi_i \rv \phi \to \lv \xi \rv ( \phi ) $ for all $ \phi \in C_c ( \Omega ) $. We first prove that for any $\tilde{\Omega}\subset\subset\Omega$,
	\begin{equation}
	\label{eq: density of discrepancy}
		\liminf_{ r \to + 0 } \frac{ \lv \xi \rv ( B_r ( x ) ) }{ r^{ n - 1 } } = 0 \text{ for all } x \in \spt \mu \cap \tilde{ \Omega }.
	\end{equation}
	We suppose the converse, that is, there exist $ x \in \spt \mu \cap \tilde{ \Omega } $, $ R > 0 $ and $ b > 0 $ such that $ R \leq \dist ( x , \partial \Omega ) ( =: r_0 ) $ and $ \lv \xi \rv ( B_r ( x ) ) \geq b r^{ n - 1 } $ for all $ 0 < r \leq R $. Using Proposition \ref{prop: density} and the definition of $ \lv \xi \rv $, we choose a large $ i $ such that
	\[
		\frac{ 1 }{ r^{ n - 1 } } \int_{ B_r ( x ) } \Bigg( \varepsilon_i \lv \nabla u_{ \varepsilon_i } \rv^2 + \frac{ \chi_{ ( - 1 , 1 ) } ( u_{ \varepsilon_i } ) }{ \varepsilon_i } \Bigg) \leq D , \quad \frac{ 1 }{ r^{ n - 1 } } \int_{ B_r ( x ) } \lv \xi^i \rv \geq \frac{ b }{ 2 }
	\]
	for all $ 0 < r \leq R $, where $ D = E_0 / r_0^{ n - 1 } > 0 $. Define $ r_1 = R \min \{ \exp ( - 4 D ) / b , 1 / 2 \} ( < R ) $. From \eqref{eq: monotonicity}, it follows that
	\begin{equation*}
		\begin{split}
			D &\geq \frac{ 1 }{ R^{ n - 1 } } \int_{ B_R ( x ) } \Bigg( \varepsilon_i \lv \nabla u_{ \varepsilon_i } \rv^2 + \frac{ \chi_{ ( - 1 , 1 ) } ( u_{ \varepsilon_i } ) }{ \varepsilon_i } \Bigg)
			\geq \int_{ r_1 }^R \frac{ 1 }{ r^n } \int_{ B_r ( x ) \cap \{ |u_{ \varepsilon_i }|<1 \} } \lv \xi^i \rv \, d y d r\\
			&\geq \frac{ b }{ 2 } \int_{ r_1 }^R \frac{ 1 }{ r } \, d r
			= \frac{ b }{ 2 } \log \Bigg( \frac{ R }{ r_1 } \Bigg) \geq 2 D,
		\end{split}
	\end{equation*}
	which is a contradiction. Thus \eqref{eq: density of discrepancy} is proven.

	Combined with Proposition \ref{prop: density} and $ \spt \lv \xi \rv \subset \spt \mu $, we have
	\[
		\liminf_{ r \to + 0 } \frac{ \lv \xi \rv ( B_r ( x ) ) }{ \mu ( B_r ( x ) ) } = 0 \text{ for all } x \in \spt \lv \xi \rv.
	\]
	A standard result in measure theory (see \cite[Lemma 1.2]{evans2015measure} for example) shows that $ \lv \xi \rv = 0 $. This concludes that $ \lv \xi_i \rv \to 0 $ in $ L^1_{ loc } ( \Omega ) $.
\end{proof}

As a consequence, we deduce that the limiting varifold is rectifiable and stationary.

\begin{prop}
\label{prop: rec}
	The limit Radon measure $ \mu $ satisfies $ \mu = \lV V \rV = \lim_{ i \to \infty } \lV V_i \rV $ and is rectifiable and stationary, that is, the first variation of $ V $ is $ 0 $.
\end{prop}
\begin{proof}
	The first part of this proposition is obvious by the definition of the weight. Note that $ \mu_i = 2 \varepsilon_i \lv \nabla u_{ \varepsilon_i } \rv^2 + \xi_i $. By the definition of $ \delta V_i $ and \eqref{eq: stationary}, we have
	\begin{equation*}
		\begin{split}
			\delta V_i ( g ) &= \int_{ \{ \lv \nabla u_{ \varepsilon_i } \rv \neq 0 \} } \bigg( \varepsilon_i \lv \nabla u_{ \varepsilon_i } \rv^2 + \frac{ \chi_{ ( - 1  , 1 ) } ( u_{ \varepsilon_i } ) }{ \varepsilon_i } \bigg) \dv g - 2 \varepsilon_i \nabla u_{ \varepsilon_i } \cdot D g \nabla u_{ \varepsilon_i }\\
			&\quad \quad - D g \cdot \bigg( \frac{ \nabla u_{ \varepsilon_i } }{ \lv \nabla u_{ \varepsilon_i } \rv } \otimes \frac{ \nabla u_{ \varepsilon_i } }{ \lv \nabla u_{ \varepsilon_i } \rv } \bigg) \xi_i\\
			&= - \int_{ \{ \lv \nabla u_{ \varepsilon_i } \rv \neq 0 \} }  D g \cdot \bigg( \frac{ \nabla u_{ \varepsilon_i } }{ \lv \nabla u_{ \varepsilon_i } \rv } \otimes \frac{ \nabla u_{ \varepsilon_i } }{ \lv \nabla u_{ \varepsilon_i } \rv } \bigg) \xi_i
		\end{split}
	\end{equation*}
	for all $ g \in C^1_c ( \Omega ; \R^n ) $. Combined with this, Proposition \ref{prop: dicrepancy vanishes} and the varifold convergence $ V_i \to V $, we obtain
	\[
		\delta V = 0,
	\]
	that is, $ V $ is stationary. Moreover, since $ \lV \delta V \rV $ is particularly a Radon measure on $ \Omega $ and the lower density estimate in Proposition \ref{prop: density} holds, we conclude that $ V $ is rectifiable by Allard's rectifiability theorem \cite[5.5 (1)]{allard1972first}.
\end{proof}

\section{Integrality of the Limit Varifold}
\label{seq: Integrality}
We let
\begin{equation}\label{eq: def of energy and discrepancy}
    e_{ \varepsilon } := \varepsilon \lv \nabla u_{ \varepsilon } \rv^2 + \frac{ \chi_{ ( - 1  , 1 ) } ( u_{ \varepsilon } ) }{ \varepsilon } , \quad \xi_{ \varepsilon } := \varepsilon \lv \nabla u_{ \varepsilon } \rv^2 - \frac{ \chi_{ ( - 1  , 1 ) } ( u_{ \varepsilon } ) }{ \varepsilon }.
\end{equation}
Define $ T : \R^n \to \R^{ n - 1 } $ by $ T ( x ) = ( x_1 , \ldots , x_{ n - 1 } ) $ and $ T^{ \perp } : \R^n \to \R $ by $ T ( x ) = x_n $, where $ x = ( x_1 , \ldots , x_{ n - 1 } , x_n ) $. Moreover, define $ \nu $ by
\[
	\nu = \begin{cases}
		\frac{ \nabla u_{ \varepsilon } }{ \lv \nabla u_{ \varepsilon } \rv } &( \lv \nabla u_{ \varepsilon } \rv \neq 0 ),\\
		0 &( \lv \nabla u_{ \varepsilon } \rv = 0 ).
	\end{cases}
\]

First, we quote the vertical monotonicity formula \cite[Lemma 5.4]{hutchinson2000convergence}. In the following lemma, $ Y $ represents a set of points on transition layers $ \{ \lv u_{ \varepsilon } \rv < 1 \} $ where $ u_{ \varepsilon } = t $. This lemma states that parallel lines to divide the points of $ Y $ can be drawn and the monotonicity formula holds for each strip domain even when transition layers are packed in a narrow region.

\begin{lemma}
\label{lem: vartically monotonicity formula}
    Suppose
    \begin{enumerate}
        \item $ N \geq 1 $ in an integer, $ Y $ is a subset of $ \R^n $, $ 0 < R < \infty $, $ 0 < \eta < 1 $, $ 0 < a < \infty $, $ 0 < \varepsilon < 1 $, $ 0 < E_0 < \infty $, $ 0 < \delta < 1 $ and $ - \infty \leq t_1 < t_2 \leq \infty $.
        \item  $ Y $ has no more than $ N + 1 $ elements, $ T ( y ) = 0 $ for all $ y \in Y $, $ Y \subset \{ \lv u_{ \varepsilon } \rv < 1 \} \cap \{ t_1 + a < x_n < t_2 - a \} $ and $ | y - x | > 3 a $ for any distinct $ x , y \in Y $.
        \item $ ( M + 1 ) \diam Y < R $ and let $ \tilde{R} := M \diam Y $.
        \item On $ \{ x \in \R^n \mid \dist ( x , Y ) < R \} $, $ u_{ \varepsilon } $ is a classical solution of $ J_\varepsilon $.
        \item For each $ x = ( x_1 , \ldots , x_n ) \in Y $,
        \[
            \int_0^R \frac{ 1 }{ \tau^n } \int_{ B_{ \tau } ( x ) \cap \{ y_n = t_j \} } \lv e_{ \varepsilon } ( y_n - x_n ) - \varepsilon \partial_n u_{ \varepsilon } ( y - x ) \cdot \nabla u_{ \varepsilon } \rv \, d \mathcal{ H }^{ n - 1 } ( y ) d \tau \leq \eta
        \]
        for $ j = 1 , 2 $ (note that in case of $t_1=-\infty$ or $t_1=\infty$, the above condition holds trivially).
        \item For each $ x \in Y 
        $ and $ a \leq r \leq R $,
        \[
            \int_{ B_r ( x ) \cap \{ \lv u_{ \varepsilon } \rv < 1 \} } \lv \xi_{ \varepsilon } \rv + ( 1 - ( \nu_n )^2 ) \varepsilon \lv \nabla u_{ \varepsilon } \rv^2 \leq \eta r^{ n - 1 } \text{ and } \int_{ B_r ( x ) \cap \{ \lv u_{ \varepsilon } \rv < 1 \} } \varepsilon \lv \nabla u_{ \varepsilon } \rv^2 \leq E_0 r^{ n - 1 }.
        \]
    \end{enumerate}
    Then the following hold:
    \begin{enumerate}
        \item There exists $ t_3 \in ( t_1 , t_2 ) $ such that for all $x\in Y$, $ \lv x_n - t_3 \rv \geq \frac{3}{2}(1-\delta) $ and 
        \begin{equation*}
        \begin{split}
            \int_0^{ \tilde{R} } \frac{ 1 }{ \tau^n } &\int_{ B_{ \tau } ( x ) \cap \{ y_n = t_3 \} } \lv e_{ \varepsilon } ( y_n - x_n ) - \varepsilon \partial_n u_{ \varepsilon } ( y - x ) \cdot \nabla u_{ \varepsilon } \rv \, d \mathcal{ H }^{ n - 1 } ( y ) d \tau\\
            &\leq \frac{( N + 1 ) N M}{\delta} \left( \eta + E_0^{ 1 / 2 } \eta^{ 1 / 2 } \right).
        \end{split}
        \end{equation*}
        \item Let
        \begin{equation*}
            \begin{split}
                &Y_1 := Y \cap \{ t_1 < x_n < t_3 \} , \quad Y_2 := Y \cap \{ t_3 < x_n < t_2 \},\\
                &\mathcal{S}_0 := \{ x \mid t_1 < x_n < t_2 \text{ and } \dist( x , Y ) < R \},\\
                &\mathcal{S}_1 := \{ x \mid t_1 < x_n < t_3 \text{ and } \dist( x , Y ) < \tilde{R} \},\\
                &\mathcal{S}_2 := \{ x \mid t_3 < x_n < t_2 \text{ and } \dist( x , Y ) < \tilde{R} \}.
            \end{split}
        \end{equation*}
        Then $ Y_1 $ and $ Y_2 $ are non-empty and for any $x\in Y$,
        \begin{equation*}
            \begin{split}
                &\frac{ 1 }{ r^{ n -1 } } \left( \int_{ \mathcal{S}_1 \cap B_r( x ) \cap \{ \lv u_{ \varepsilon } \rv < 1 \} } e_{ \varepsilon } + \int_{ \mathcal{S}_2 \cap B_r( x ) \cap \{ \lv u_{ \varepsilon } \rv < 1 \} } e_{ \varepsilon } \right)\\
                &\leq \left( 1 + \frac{ 1 }{ M } \right)^{ n - 1 } \frac{ 1 }{ R^{ n - 1 } } \int_{ \mathcal{ S }_0 \cap \{ \lv u_{ \varepsilon } \rv < 1 \} } e_{ \varepsilon } + c ( n , N , M ) \left( \eta + E_0^{ 1 / 2 } \eta^{ 1 / 2 } \right)
            \end{split}
        \end{equation*}
        holds for any $ a \leq t < \tilde{ R } $.
    \end{enumerate}
\end{lemma}
Starting $ t_1 = - \infty $ and $ t_2 = \infty $, we inductively apply Lemma \ref{lem: vartically monotonicity formula} to divide all the element of $ Y $. Then, by choosing $ M $ sufficiently large and taking $ \eta $ sufficiently small depending on $ M $ and $ N $, we have the following lemma, which demonstrates that the whole energy density can be estimated from the sum of the energy densities in the strip region.
\begin{lemma}
    \label{lem: sheet separation 2}
    Given $ 0 < R < \infty $, $ 0 < E_0 < \infty $, $ 0 < s < 1 $ and $ N \in \N $, there exists $ \eta > 0 $ with the following property: Assume
    \begin{enumerate}
        \item $ Y $ is a subset of $ \R^n $, $ Y $ has $ N $ elements, $ T ( y ) = 0 $ for all $ y \in T $, $ 0 < a < 1 $, $ \lv y - z \rv > 3 a $ for all any distinct $ y , z \in Y $ and $ \diam Y \leq \eta R $.
        \item On $ \{ x \in \R^n \mid \dist ( x , Y ) < R \} $, $ u_{ \varepsilon } $ is a classical solution of $ J_\varepsilon $.
        \item For each $ x \in Y 
        $ and $ a \leq r \leq R $,
        \[
            \int_{ B_r ( x ) \cap \{ \lv u_{ \varepsilon } \rv < 1 \} } \lv \xi_{ \varepsilon } \rv + ( 1 - ( \nu_n )^2 ) \varepsilon \lv \nabla u_{ \varepsilon } \rv^2 \leq \eta r^{ n - 1 } \text{ and } \int_{ B_r ( x ) \cap \{ \lv u_{ \varepsilon } \rv < 1 \} } \varepsilon \lv \nabla u_{ \varepsilon } \rv^2 \leq E_0 r^{ n - 1 }.
        \]
    \end{enumerate}
    Then there exist $ - \infty = t_{ 0 , - } < t_1 < \cdots < t_{ N - 1 } < t_{ 0 , + } = \infty $ such that, for each $ y \in Y $, we can choose a sheet $ S_y = \{ x \mid t_{ j - 1 } < x_n < t_j \} $ so that $ S_y \cap Y = \{ y \} $ and 
    \begin{equation}
    \label{eq: sheet sparated monotonicity}
        \sum_{ y \in Y } \frac{ 1 }{ r^{ n - 1 } } \int_{ \mathcal{ S }_y \cap B_r ( y ) \cap \{ \lv u_{ \varepsilon } \rv < 1 \} } e_{ \varepsilon } \leq s + \frac{ 1 + s }{ R^{ n - 1 } } \int_{ \{ x \mid \dist( x , Y ) < R \} } e_{ \varepsilon }
    \end{equation}
    holds for any $ a \leq r < \eta R $.
\end{lemma}

We next show that the smallness of the discrepancy and tilt excess imply that the solution is close to a one-dimensional solution in an $\varepsilon$-scale ball.

\begin{prop}
\label{prop: flatness}
	Given $ 0 < s < 1 $, there exist $ 0 < \eta < 1/4 $, $ 0 < b < 1 $ and $ 1 < L < \infty $ with the following property: Assume $ u_{\varepsilon}(0)=0 $, $ 0 < \varepsilon < 1 $ and $ u_{ \varepsilon } $ is a classical solution of $ J_\varepsilon $, and
	\begin{equation}
	\label{eq: C1 L1 flatness}
		\int_{ B_{ 2 L \varepsilon } ( 0 )\cap\{|u|<1\}} \big( \lv \xi_{ \varepsilon } \rv + ( 1 - ( \nu_n )^2 ) \varepsilon \lv \nabla u_{ \varepsilon } \rv^2 \big) \leq \eta ( 2 L \varepsilon )^{ n - 1 }.
	\end{equation}
	Moreover, we assume that $\{|u_\varepsilon|<1\}\cap B_{2L\varepsilon}(0)$ is connected and
    \begin{equation}
    \label{eq: C2 uniform bound assumption}
        \lV D^2 u_{\varepsilon} \rV_{ L^{\infty} (B_{2L\varepsilon}(0))} \leq \frac{ C }{ \varepsilon^2 }
    \end{equation}
    for some constant $ C > 0 $ independent of $\varepsilon$. We then have
	\begin{equation}
	\label{eq: energy density 1}
		\bigg\lv \frac{ 1 }{ \omega_{ n - 1 } ( L \varepsilon )^{ n - 1 } } \int_{ B_{ L \varepsilon } ( 0 ) \cap \{|u_\varepsilon|<1\} \cap\{|x_n|\leq(1-b)\varepsilon\}} e_{ \varepsilon } - 4 \bigg\rv \leq s.
	\end{equation}
\end{prop}

\begin{proof}
    We rescale the domain by $\varepsilon$ for convenience. Thus we may assume $\varepsilon=1$, and drop $\varepsilon$ from the notation in the following. 

    We first show the claim without $b$, i.e. 
    \begin{equation}\label{eq: energy density without b}
        	\bigg\lv \frac{ 1 }{ \omega_{ n - 1 } ( L \varepsilon )^{ n - 1 } } \int_{ B_{ L \varepsilon } ( 0 ) \cap \{|u_\varepsilon|<1\}  } e_{ \varepsilon } - 4 \bigg\rv \leq s.
    \end{equation}
    Consider the truncated one-dimensional profile
    \[
        q(x)=q(x',x_n)=\max\{\min\{x_n+u(0),1\},-1\}.
    \]
    Then take $L>0$ large enough, so that
    \begin{equation}\label{eq: energy estimate for 1 dimensional profile}
            \begin{split}
            &\left|\frac{1}{\omega_{n-1}L^{n-1}}\int_{B_{L}(0)\cap\{|q|<1\}}\left(|\nabla q|^2+\chi_{(-1,1)}(q)\right)-4\right|
            = 2\left(\frac{|B_L(0)\cap\{|q|<1\}|}{\omega_{n-1}L^{n-1}}-2\right)
            \leq \frac{s}{2}.
        \end{split}
    \end{equation}
    Let $ f = 1 - (\partial_n u )^2 $. Take $x_0\in B_L(0)\cap\{|u|<1\}$, such that
    \[
        \|f\|_{L^\infty(B_{L}\cap\{|u|<1\})}\leq 2f(x_0).
    \]
    For any $x\in B_r(x_0)\cap\{|u|<1\}$ (we choose $r$ later), by \eqref{eq: C2 uniform bound assumption}, we have
    \[
        |f(x_0)|\leq |f(x)|+r\|f\|_{C^1(B_{3L}\cap\{|u|<1\})}\leq |f(x)|+Cr.
    \]
    Therefore, integrate above in $B_r(x_0)\cap\{|u|<1\}$ and using
    \[
        f=1-(\partial_nu)^2=|\xi|+(1-\nu_n^2)|\nabla u|^2,
   \]
    we obtain
    \[
        |f(x_0)|\leq C(\eta r^{-n}+r).
    \]
    By taking $r=n^{\frac{1}{n+1}}\eta^{\frac{1}{n+1}}$, we conclude $|f(x_0)|\leq C\eta^{\frac{1}{n+1}}$. This implies that
    \[
       \|1-(\partial_nu)^2\|_{L^\infty(B_{L}(0)\cap\{|u|<1\})}=\|f\|_{L^\infty(B_{L}(0)\cap\{|u|<1\})}\leq C\eta^{\frac{1}{n+1}}.
    \]
    From this, we conclude that $u$ is $C^1$ close to $q(x_n)$ on $B_{L}(0)\cap\{|u|<1\}$. Take $\eta$ sufficiently small, the proof of \eqref{eq: energy density without b} follows from \eqref{eq: energy estimate for 1 dimensional profile}.

    Finally, we show that we can take small $b$ (depends on $n$ and $s$) so that \eqref{eq: energy density 1} follows. From the $C^1$ flatness of $u$, we can estimate $\|h_t-x_n\|_{L^\infty(B_L(0)\cap\{|u|<1\})}$ in algebraic order of $\eta$, where $h_t$ denotes the function whose graph is the $t$-level surface of $u$. Therefore, by taking $b$ small (independent of $\eta<1/4$), we have
    \[
        B_L(0)\cap\{|x_n|>1-b\}\subset B_L(0)\cap \{|u|>1-b\}.
    \]
    Therefore, by taking $b$ small compared to $s$, we can estimate the right hand side of
    \[
        \int_{B_L(0)\cap\{|x_n|>1-b\}}e_\varepsilon\leq\int_{B_L(0)\cap\{|u|>1-b\}}e_\varepsilon
    \]
    as small as we want. Then we complete the proof of \eqref{eq: energy density 1} from \eqref{eq: energy density without b}.
    \end{proof}

\begin{theorem}
\label{thm: integrality}
	Under the uniform $C^2$ bound assumption $ \lV D^2 u \rV_{ L^{\infty } } \leq C / \varepsilon^2 $, the varifold $ V $ as in Proposition \ref{prop: rec} is integral.
\end{theorem}
\begin{proof}
	Since $ V $ is rectifiable, $ V $ has an approximate tangent plane for $ \mathcal{ H }^{ n - 1 } $-almost everywhere on $ \spt \lV V \rV $. We may assume that there is the approximate tangent plane at the origin and choose coordinates so that the tangent plane is $ T = \{ x \in \R^n \mid x_n = 0 \} $. We set $ \Phi_r ( x ) = x / r $ and take a sequence $ \{ r_i \}_{ i \in \N } $ so that $ r_i \to 0 $ and $ ( \Phi_{ r_i } )_{ \# } V \to \theta \, \lv T \rv $ in the varifold sense, where $ \theta $ denotes the density of $ V $ at the origin and $ ( \Phi_r )_{ \# } $ is the usual push-forward. By taking a subsequence if necessary, we may assume $ \lim_{ i \to \infty } ( \Phi_{ r_i } )_{ \# } V_i = \theta \, \lv T \rv $ and $ \varepsilon_i / r_i \to 0 $. Let $ u_i ( x ) = u_{ \varepsilon_i } ( r_i x ) $. We observe that $ u_i $ is the solution of the following
	\[
		\begin{cases}
			\Delta u_i = 0 &\text{in } \{ \lv \tilde{ u }_i \rv < 1 \}\\
			\lv \nabla u_i \rv = 1 / \tilde{ \varepsilon }_i &\text{on } \partial \{ \lv \tilde{ u }_i \rv < 1 \}
		\end{cases}
	\]
	with $ \tilde{ \varepsilon }_i = \varepsilon_i / r_i \to 0 $. In the following, we abuse the notation and use $\varepsilon_i$ in place of $\tilde{\varepsilon}_i$, and write $ V_i $ and $\xi_i$, for the varifold and the discrepancy associated to $ u_i $, respectively.
	
	By Proposition \ref{prop: dicrepancy vanishes}, we can see that
	\begin{equation}
	\label{eq: discrepancy limit}
		\lim_{ i \to \infty } \int_{ B_3 ( 0 ) \cap \{ \lv u_i \rv < 1 \} } |\xi_i| = 0,
	\end{equation}
	and the Radon measure $ 2 \lv \nabla u_i \rv \, d x $ converges to the same limit $ \theta \, \lv T \rv $. Since $ V_i \to \theta \, \lv T \rv $ in the varifold sense, we have
	\begin{equation}
	\label{eq: flatness limit}
		\lim_{ i \to \infty } \int_{ B_3 ( 0 ) \cap \{ \lv u_i \rv < 1 \} } \big( 1 - ( \nu_n )^2 \big) \varepsilon_i \lv \nabla u_i \rv^2 = 0.
	\end{equation}

	Let $ N $ be the smallest positive integer greater than $ \theta / 4 $ and let $ s > 0 $ be arbitrary small. Corresponding to $ s $, we choose $ \eta $ and $ L $ by Proposition \ref{prop: flatness}. We also restrict $ \eta $ so that $ 1 / ( 1 - \eta )^{ n - 1 } \leq 1 + s $. For all large $ i $, we define 
	\begin{equation*}
		\begin{split}
			G_i &= B_2 ( 0 ) \cap \{ \lv u_i \rv < 1 \} \cap\\
			&\bigg\{ x \mid \int_{ B_r ( x ) \cap \{ \lv u_i \rv < 1 \} } \lv \xi_{ \varepsilon_i } \rv + \big( 1 - ( \nu_n )^2 \big) \varepsilon_i \lv \nabla u_i \rv^2 \leq \eta r^{ n - 1 }, \text{ for all }r\in[2L\varepsilon_i,1] \bigg\}.
		\end{split}
	\end{equation*}
    Combined with Proposition \ref{prop: density} and the Besicovitch covering theorem, one shows
	\begin{equation*}
		\begin{split}
			&\lV V_i \rV ( B_2 ( 0 ) \cap \{ \lv u_i \rv < 1 \} \setminus G_i ) + \mathcal{L}^{n-1}( T ( B_2 ( 0 ) \cap \{ \lv u_i \rv < 1 \} \setminus G_i))\\
            &\leq c \eta^{ - 1 } \int_{ B_3 ( 0 ) } \lv \xi_{ \varepsilon_i } \rv + \big( 1 - ( \nu_n )^2 \big) \varepsilon_i \lv \nabla u_i \rv^2
		\end{split}
	\end{equation*}
	for some $ c = c ( n ) > 0 $. Thus, by \eqref{eq: discrepancy limit} and \eqref{eq: flatness limit}, we obtain
	\begin{equation}
	\label{eq: Gi convergence}
		\lim_{ i \to \infty } \lV V_i \rV ( B_2 ( 0 ) \cap \{ \lv u_i \rv < 1 \} \setminus G_i ) = 0 \text{ and } \lim_{i\to\infty}\dist(T,G_i)=0.
	\end{equation}
    
    We let $ Y = B_1 ( 0 ) \cap T^{ - 1 } ( x ) \cap G_i \cap \{ u_i = t \} $ for $ x \in B_1^{ n - 1 } ( 0 ) := ( \R^{ n - 1 } \times \{ 0 \} ) \cap B_1 ( 0 ) $ and $ t \in ( - 1 , 1 ) $. By Proposition \ref{prop: flatness}, each element $ y $ of $ Y $ is contained a distinct connected component $ U_y $ of $ \{ \lv u_i \rv < 1 \} \cap \{ ( x' , x_n ) \in B_1^{ n - 1 } ( 0 ) \times \R \mid \lv x' \rv \leq L \varepsilon_i \} \cap B_1 ( 0 ) $. Suppose $|Y|\geq N$. We construct $\tilde{Y}$ from $Y$ in the following manner: we pick arbitrary $N$ points in $Y$, then move these points to the closest nodal set $\{u_i=0\}$ in $x_n$ direction. Precisely, let $Y^*$ be a family of $N$ points, arbitrary chosen from $Y$. Then (where $\mathbf{e}_n:=(0,\dots,0,1)$ is the unit vector in $n$ direction)
    \[
        \tilde{Y}:=\left\{y-\mathbf{e}_nd_\textup{N}(y):y\in Y^*\right\}\subset\{u_i=0\},
    \]
    where $d_\textup{N}(y)$ is the signed distance in $x_n$ direction, from the nodal set $\{u_i=0\}$. Then, as in the proof Proposition \ref{prop: flatness}, we know that $u_i$ is $C^1$ flat around $y\in Y\subset G_i$, thus the following holds:
    \begin{enumerate}
        \item $ | \tilde{ Y } | = N $,
        \item $ \tilde{ Y } \subset B_1 ( 0 ) \cap T^{ - 1 } ( x ) \cap G_i $,
        \item $ | \tilde{y} - \tilde{z} | \geq 2 \varepsilon $ for $ \tilde{y} , \tilde{z} \in \tilde{Y} $ with $ \tilde{y} = \tilde{z} $,
        \item for all $ \tilde{y} \in \tilde{Y} $, there exists $ y \in Y $ such that $ \tilde{y} \in U_y $,
        \item \eqref{eq: energy density 1} holds for each element of $ \tilde{Y} $.
    \end{enumerate}
    Applying Lemma \ref{lem: sheet separation 2} with $ \delta = b / 2 $ to this $ \tilde{Y} $, there exist $ - \infty = t_{ 0 , - } < t_1 < \cdots < t_{ N - 1 } < t_{ 0 , + } = \infty $ such that $ t_j - t_{ j - 1 } \geq ( 1 - b / 2 )( 2 - b ) \varepsilon $ and \eqref{eq: sheet sparated monotonicity} holds. By Proposition \ref{prop: flatness},
    \[
        4 - s \leq \frac{ 1 }{ \omega_{ n - 1 } ( L \varepsilon_i )^{ n - 1 } } \int_{ U_y \cap B_{ L \varepsilon_i } ( \tilde{ y } ) \cap \{ \lv x_n - \tilde{ y }_n \rv \leq 2 ( 1 - b ) \varepsilon \} } e_{ \varepsilon_i }
    \]
    for each $ \tilde{ y } \in \tilde{ Y } $, where $ U_y $ is a connected component corresponding to $ \tilde{ y } $. Since $ V_i \to \theta \, \lv T \rv $ in $ B_3 ( 0 ) $,
	\[
		\sup_{ x \in B_1^{ n - 1 } ( 0 ) } \frac{ 1 }{ \omega_{ n - 1 } } \int_{ B_1 ( x ) } e_{ \varepsilon_i } \leq \theta + s
	\]
	holds for sufficiently large $ i $. Due to \eqref{eq: Gi convergence} and using $ V_i \to \theta \, \lv T \rv $ again, we also have $ \diam Y \leq \eta $ for large $ i $. By Lemma \ref{lem: sheet separation 2} and the above two inequality, we would have
	\[
		4 N \leq N s + ( 1 + s ) ( \theta + s ).
	\]
	This would be a contradiction to $ \theta < 4 N $ for sufficiently small $ s $ depending only on $ N $.
	
	From \eqref{eq: Gi convergence} and the fact that $ 2 \lv \nabla u_i \rv \, d x $ converges to $ \lV V \rV $, it follows that
	\begin{equation}
	\label{eq: 1}
		\lim_{ i \to \infty } \int_{ B_1 ( 0 ) \cap \{ \lv u_i \rv < 1 \} \cap G_i } e_{ \varepsilon_i } = \lim_{ i \to \infty } \int_{ B_1 ( 0 ) \cap \{ \lv u_i \rv < 1 \} \cap G_i } 2 \lv \nabla u_i \rv = \lim_{ i \to \infty } \int_{ B_1 ( 0 ) \cap \{ \lv u_i \rv < 1 \} } 2 \lv \nabla u_i \rv.
	\end{equation}
	By \eqref{eq: flatness limit}, \eqref{eq: Gi convergence} and Proposition \ref{prop: dicrepancy vanishes}, we have
	\begin{equation}
		\label{eq: 2}
		\lim_{ i \to \infty } \int_{ B_1 ( 0 ) \cap \{ \lv u_i \rv < 1 \} \cap G_i } 2 \lv \nabla u_i \rv = \lim_{ i \to \infty } \int_{ B_1 ( 0 ) \cap \{ \lv u_i \rv < 1 \} \cap G_i } 2 \lv \nu_n \rv \lv \nabla u_i \rv.
	\end{equation}
	Using the coarea formula and the area formula (see \cite[10.6 and 12.4]{simon1983lectures} for example), we see that
	\begin{equation}
	\label{eq: 3}
		\begin{split}
			&\int_{ B_1 ( 0 ) \cap \{ \lv u_i \rv \leq 1 \} \cap G_i } 2 \lv \nu_n \rv \lv \nabla u_i \rv
			= \int_{ - 1 }^{ 1 } \int_{ B_1 ( 0 ) \cap \{ u_i = t \} \cap G_i } 2 \lv \nu_n \rv \, d \mathcal{ H }^{ n - 1 } d t\\
			&= \int_{ - 1 }^{ 1 } \int_{ \{ x_n = 0 \} } 2 \mathcal{ H }^0 ( B_1 ( 0 ) \cap \{ u_i = t \} \cap G_i \cap T^{ - 1 } ( x ) ) \, d \mathcal{ H }^{ n - 1 } ( x ) d t. 
		\end{split}
	\end{equation}
	Therefore, noting that the number of elements of $ B_1 ( 0 ) \cap \{ u_i = t \} \cap G_i \cap T^{ - 1 } ( x ) $ is less than $ N - 1 $ for all $ x \in B_1^{ n - 1 } ( 0 ) $ and all $ t \in ( - 1 , 1 ) $ and thanks to \eqref{eq: 1}-\eqref{eq: 3}, we obtain
	\begin{equation*}
		\begin{split}
			\omega_{ n - 1 } \theta
			&= \lV \theta \, \lv T \rv \rV ( B_1 ( 0 ) )
			= \lim_{ i \to \infty } \bigg( \int_{ B_1 ( 0 ) \cap \{ \lv u_i \rv < 1 \} } e_{ \varepsilon_i } \bigg)\\
			&= \lim_{ i \to \infty } \int_{ - 1 }^{ 1 } \int_{ \{ x_n = 0 \} } 2 \mathcal{ H }^0 ( B_1 ( 0 ) \cap \{ u_i = t \} \cap G_i \cap T^{ - 1 } ( x ) ) \, d \mathcal{ H }^{ n - 1 } ( x ) d t\\
			&\leq 4 \omega_{ n - 1 } ( N - 1 ).
		\end{split}
	\end{equation*}
	Since $ N $ is the smallest positive integer greater than $ \theta / 4 $, $ \theta = 4 ( N - 1 ) $ holds. This completes the proof.
\end{proof}

\section{Proofs of main results}
\label{sec: proofs of main}
\subsection{Proof of Theorem \ref{thm: main-euclidean}}
    We collect the results above to complete the proof of Theorem \ref{thm: main-euclidean}. (1) is shown in Proposition \ref{prop: rec}. (2) is precisely Corollary \ref{cor: uniform}. (3) follows form a contradiction argument from Proposition \ref{prop: density} and Corollary \ref{cor: uniform}. We finally prove (4). Let us use the same notation as in the proof of Theorem \ref{thm: integrality}. We note that $ u_i $ converges locally uniformly to $ + 1 $ on one side of $ T $ and $ - 1 $ on the other side at $ \mathcal{ H }^{ n - 1 } $-a.e.\,$ x \in \partial^* \{ u_0 = 1 \} $, and to the same value for $ \mathcal{ H }^{ n - 1 } $-a.e.\,$ x \in \Omega \setminus \partial^* \{ u_0 = 1 \} $. We also note that we may choose $ x_i \in B_1^{ n - 1 } ( 0 ) $ and $ t \in ( - 1 , 1 ) $ such that $ x_i \in T ( \{ \lv u_i \rv < 1 \} \cap ( B_2 ( 0 ) \setminus G_i ) ) $ and $ T^{ - 1 } ( x_i ) \cap G_i \cap \{ u_i = t \} $ has precisely $ N - 1 $ elements. We thus see that $ T^{ - 1 } ( x_i ) \cap \{ u_i = t \} $ has $ N - 1 $ elements. This observation immediately implies that the density is either odd or even, depending on the sign of $ u_i $ away from $ T $, and this distinction corresponds to whether the origin is in $ \partial^* \{ u_0 = 1 \} $ or not.

\subsection{Convergence of Minimizers: \texorpdfstring{$\Gamma$}{}-convergence}\label{subsec: Gamma convergence}
In this subsection, we show that the minimizers of $ J_\varepsilon $ converge to a minimizer of $ J_0 $ defined by
\[
    J_0 ( u ) := \begin{cases}
        2 \int_\Omega \lv \nabla u \rv &(\text{if } u \in BV ( \Omega ; \{ \pm 1 \} ) ),\\
        + \infty &(\text{otherwise}),
    \end{cases}
\]
which is the area functional in $\Omega$. To this end, we prove Theorem \ref{thm: Gamma convergence}, well known as the $ \Gamma $-convergence by De Giorgi. Once Theorem \ref{thm: Gamma convergence} is obtained, the general theory of the $ \Gamma $-convergence allows to reach the desired conclusion.

In order to prove Theorem \ref{thm: Gamma convergence}, we need the following lemma coming from the general facts of the set of finite perimeter (see \cite[Lemma 1 and Lemma 2]{sternberg1988effect}, for example).

\begin{lemma}
\label{lem: smooth appro BV}
    Let $ \Omega \subset \mathbb{R}^n$ be an bounded open set with smooth boundary. Then the following hold.
    \begin{enumerate}
        \item Let $ A \subset \Omega $ be a set of finite perimeter in $ \Omega $ with $ 0 < \lv A \rv < \lv \Omega \rv $. Then there exists a sequence of open sets $ \{ A_k \}_{ k \in \N } $ satisfying the following:
        \begin{enumerate}
            \item $ \partial A_k \cap \Omega $ is $ C^2 $ for all $ k $,
            \item $ \lim_{ k \to \infty } \lV \chi_{ A_k \cap \Omega } - \chi_A \rV_{ L^1 ( \Omega ) } = 0 $,
            \item $ \lim_{ k \to \infty } \mathcal{ H }^{ n - 1 } ( \partial A_k ) = \mathcal{ H }^{ n - 1 } ( \partial^* A ) $,
            \item $ \mathcal{ H }^{ n - 1 } ( \partial A_k \cap \partial \Omega ) = 0 $ for all $ k $.
        \end{enumerate}
        \item Let $ A \subset \Omega $ be an open set with $ C^2 $, compact, nonempty boundary such that
        \[
            \mathcal{ H }^{ n - 1 } ( \partial A \cap \partial \Omega ) = 0.
        \]
        Define the signed distance function of $ \partial A $ by
        \[
            d_{ \partial A } ( x ) := \begin{cases}
                \dist ( x , \partial A ) &( x \in \Omega \setminus A ),\\
                - \dist ( x , \partial A ) &( x \in A \cap \Omega ).
            \end{cases}
        \]
        Then, for sufficiently small $ \delta > 0 $, $ d $ is a $ C^2 $ function in $ \{ \lv d \rv < \delta \} $ with $ \lv \nabla d \rv = 1 $ and we have
        \begin{equation}
        \label{eq: convergence of sd level area}
            \lim_{ \delta \to 0 } \mathcal{ H }^{ n - 1 } ( \{ d_{ \partial A } = \delta  \} ) = \mathcal{ H }^{ n - 1 } ( \partial A ).
        \end{equation}
    \end{enumerate}
\end{lemma}

\begin{proof}[Proof of Theorem \ref{thm: Gamma convergence}]
(1) follows from the definition of $J_0$, Assumption \ref{assm: energy} and \eqref{eq: AMGM for J}.

We prove (2). Take $ u \in L^1 ( \Omega ; [ - 1 , 1 ] ) $, and write $ u = 2 \chi_A - 1 $ for some set of finite perimeter $ A \subset \Omega $. If $u$ is of other form, then $J_0(u)=+\infty$ from the definition, and the claim is trivial. It is sufficient to prove that there exists a sequence $ \{ u_{ \varepsilon_i } \}_{ i \in \N } \subset W^{ 1 , 2 } ( \Omega ; [ - 1 , 1 ] ) $ such that
\begin{equation}
\label{eq: desired conculusion (2)}
    \lim_{ i \to \infty } \lV u_{ \varepsilon_i } - u \rV_{ L^1 ( \Omega ) } = 0 \text{ and } 2 \int_\Omega \lv \nabla u \rv = 4 \mathcal{ H }^{ n - 1 } ( \partial^* A ) = \lim_{ i \to \infty } J_{\varepsilon_i} ( u_{ \varepsilon_i } ),
\end{equation}
for some $ \{ \varepsilon_i \}_{ i \in \N } $ with $ \varepsilon_i = 0 $ as $ i \to \infty $. If $ \lv A \rv = \lv \Omega \rv $ or $ 0 $, then taking $ u_{ \varepsilon } = 1 $ or $ - 1 $ is sufficient. Thus, by Lemma \ref{lem: smooth appro BV} (1) and the diagonal argument, it is sufficient to prove \eqref{eq: desired conculusion (2)} when $ \partial A $ is $ C^2 $ taking a subsequence if necessary.

For $ \partial A $, we define $ u_{ \varepsilon } $ by
\[
    u_{ \varepsilon } ( x ) := \begin{cases}
        \varepsilon^{ - 1 } d_{ \partial A } ( x ) &( x \in \{ \lv d_{ \partial A } \rv < \varepsilon \} ),\\
        - 1 &( x \in \{ d_{ \partial A } \geq \varepsilon \} ),\\
        1 &( x \in \{ d_{ \partial A } \leq - \varepsilon \} ).
    \end{cases}
\]
Note that, for sufficiently small $ \varepsilon > 0 $, $ u_{ \varepsilon } $ is a Lipshitz function and $ \lv \nabla d_{ \partial A } \rv = 1 $ in $ \{ \lv d_{ \partial A } \rv < \varepsilon \} $. Therefore, by the coarea formula, we have
\begin{equation*}
    \begin{split}
        J_\varepsilon ( u_{ \varepsilon } ) &= \int_\Omega \varepsilon \lv \nabla u_{ \varepsilon } \rv^2 + \frac{ \chi_{ ( - 1 , 1 ) } ( u_{ \varepsilon } ) }{ \varepsilon } = \int_{ - \varepsilon }^{ \varepsilon } \int_{ \{ d_{ \partial A } = t \} } \frac{ 2 }{ \varepsilon } \, d \mathcal{ H }^{ n - 1 } d t\\
        &= 2 \int_{ - 1 }^1 \mathcal{ H }^{ n - 1 } ( \{ d_{ \partial A } = \varepsilon t \} ) \, d t.
    \end{split}
\end{equation*}
By \eqref{eq: convergence of sd level area} and the dominated convergence theorem, we have
\[
    \lim_{ \varepsilon \to 0 } J_\varepsilon ( u_{ \varepsilon } ) = 4 \mathcal{ H }^{ n - 1 } ( \partial A ).
\]
By the dominated convergence theorem again, we also have $ \lim_{ \varepsilon \to + 0 } \lV u_{ \varepsilon } - u \rV_{ L^1 ( \Omega ) } = 0 $. Hence we obtain \eqref{eq: desired conculusion (2)} and this completes the proof.
\end{proof}

\begin{proof}[Proof of Theorem \ref{thm: min to min}]
	If $ u_{ \varepsilon_i } $ converges to either $1$ or $-1$ uniformly, the conclusion is trivial. Thus, we may assume that $ u_{ \varepsilon_i } $ does not converge to a constant function $ + 1 $ or $ - 1 $ uniformly.
	
	First, we show that $ 4^{ - 1 } \lV V \rV = \lv \partial \{ u_0 = 1 \} \rv $. Let $ x \in \spt \lV V \rV \cap \Omega $ be a point with an approximate tangent plane. We also denote $V_i$ to be varifold associated to $u_{\varepsilon_i}$, and $\theta$ the density of $V$. Consider suitable rescaling and translations so that $ x = 0 $, $ V_i \to \theta \lv T \rv $ and $ T = \R^{ n - 1 } \times \{ 0 \} $. Since $ u_{ \varepsilon_i } $ does not converge to a constant function $ \pm 1 $ uniformly, we see that $ \{ \lv u_{ \varepsilon_i } \rv < 1 \} $ has at least one connected component. By Corollary \ref{cor: uniform}, $ u_{ \varepsilon_i } $ converges locally uniformly to either $ \pm 1 $ on $ \{ x_n > 0 \} $ and $ \{ x_n < 0 \} $. If $ u_{ \varepsilon_i } $ converges to $ + 1 $ in both $ \{ x_n > 0 \} $ and $ \{ x_n < 0 \} $, then one can compare $ u_{ \varepsilon_i } $ to $ + 1 $ on $ B_1 ( 0 ) $ and reduce the energy $J_\varepsilon$ by a definite amount, which would be a contradiction to the local energy minimality of $ u_{ \varepsilon_i } $. Hence, $ u_{ \varepsilon_i } $ converges to $ + 1 $ on one side and $ - 1 $ on the other side, which implies that $ \spt \lV V \rV = \spt \lv \partial \{ u_0 = 1 \} \rv $. Therefore, by Proposition \ref{prop: flatness}, we have $ \theta = 1 $, which completes the claim.
	
	We next use a contradiction argument for the minimality of $ u_0 $. Suppose that there exists a function $ \tilde{ u } \in BV ( \Omega ; \{ \pm 1 \} ) $ such that $ \int_\Omega \lv u_0 - \tilde{ u } \rv < c $ and $ J_0 ( \tilde{ u } ) < J_0 ( u_0 ) $. By Theorem \ref{thm: Gamma convergence} (2), there exists a sequence $ \{ \tilde{ u }_{ \varepsilon_i } \} \subset W^{ 1 , 2 } ( \Omega ; [ - 1 , 1 ] ) $ such that it follows that
	\begin{equation*}
	\label{eq: a}
		\lim_{ i \to \infty } \lV \tilde{ u }_{ \varepsilon_i } - \tilde{ u } \lV_{ L^1 ( \Omega ) } = 0 \text{ and } \limsup_{ i \to \infty } J_{\varepsilon_i} ( \tilde{ u }_{ \varepsilon_i } ) \leq J_0 ( \tilde{ u } ). 
	\end{equation*}
	Note that, since
	\[
		\int_\Omega \lv u_{ \varepsilon_i } - \tilde{ u }_{ \varepsilon_i } \rv \leq \int_\Omega \lv u_{ \varepsilon_i } - u_0 \rv + \lv u_0 - \tilde{ u } \rv + \lv \tilde{ u } - \tilde{ u }_{ \varepsilon_i } \rv
	\]
	holds, we have $ \int_\Omega \lv u_{ \varepsilon_i } - \tilde{ u }_{ \varepsilon_i } \rv < c $ for sufficiently large $ i $. By the local minimality of $ u_{ \varepsilon_i } $,
	\begin{equation*}
	\label{eq: b}
		J_{\varepsilon_i} ( u_{ \varepsilon_i } ) \leq J_{\varepsilon_i} ( \tilde{ u }_{ \varepsilon_i } )
	\end{equation*}
	holds for sufficiently large $ i $. From Theorem \ref{thm: Gamma convergence} (1), it follows that
	\begin{equation*}
		J_0 ( u_0 ) \leq \liminf_{ i \to \infty } J_{\varepsilon_i} ( u_{ \varepsilon_i } ).
	\end{equation*}
	The above equations lead to a contradiction and therefore this completes the proof.
\end{proof}

\appendix

\section{Interpolation between \texorpdfstring{$ L^1 $}{} and Lip.}
\label{app: classical}

\begin{prop}
\label{prop: interpolation L1 Lip}
	Let $ u : B_1 ( 0 ) \to \R $ be a Lipshitz function. Then
	\[
		\lV u \rV_{ L^{ \infty } ( B_R ( 0 ) ) } \leq C \max\Big\{ \lV u \rV_{ L^1 ( B_1 ( 0 ) ) }^{ 1 / ( n + 1 ) }\lV \nabla u \rV_{ L^{ \infty } ( B_1 ( 0 ) ) }^{ n / ( n + 1 ) }, \lV u \rV_{ L^1 ( B_1 ( 0 ) ) } \Big\}
	\]
	for some $ C > 0 $ depending only on $ n $.
\end{prop}
\begin{proof}
	Let $ x_0 \in B_1 ( 0 ) $ be such that $ \lv u ( x_0 ) \rv \geq \lV u \rV_{ L^{ \infty } ( B_1 ( 0 ) ) } / 2 $. We then have
    \[
        \lv u ( x ) \rv \geq \frac{ \lV u \rV_{ L^{ \infty } ( B_1 ( 0 ) ) } }{ 2 } - \lV \nabla u \rV_{ L^{ \infty } ( B_1 ( 0 ) ) } \lv x - x_0 \rv.
    \]    
    Therefore, letting
    \[
        r = \min\bigg\{ \frac{ \lV u \rV_{ L^{ \infty } ( B_1 ( 0 ) ) } }{ 16 \lV \nabla u \rV_{ L^{ \infty } ( B_1 ( 0 ) ) } }, \frac{ 1 }{ 4 } \bigg\},
    \]
    we have $ \lv u \rv \geq \lV u \rV_{ L^{ \infty } ( B_1 ( 0 ) ) } / 4 $ in $ B_r ( x_0 ) $. Hence, we obtain
	\begin{equation*}
    \begin{split}
		\int_{ B_1 ( 0 ) } \lv u \rv &\geq \int_{ B_1 ( 0 ) \cap B_r ( x_0 ) } \lv u \rv \geq \frac{ \lV u \rV_{ L^{ \infty } ( B_1 ( 0 ) ) } }{ 4 } \mathcal{ L }^n ( B_1 ( 0 ) \cap B_r ( x_0 ) )\\
        &\geq C \min\bigg\{\frac{ \lV u \rV_{ L^{ \infty } ( B_1 ( 0 ) ) }^{ n + 1 } }{ \lV \nabla u \rV_{ L^{ \infty } ( B_1 ( 0 ) ) }^n }, \frac{ 1 }{ 4^n } \lV u \rV_{ L^{ \infty } ( B_1 ( 0 ) ) } \bigg\}.
    \end{split}
	\end{equation*}
	This completes the proof.
\end{proof}

\section*{Acknowledgement}

The first author thanks Dr. Hardy Chan for careful advices. The first author have received funding from the Swiss National Science Foundation
under Grant PZ00P2\_202012/1.
The second author acknowledges the support of the Grand-in-Aid for JSPS Fellows, Grant number 25KJ1245.

\bibliography{FBAC.bib}
\bibliographystyle{abbrv}

\end{document}